\documentclass{amsart}
\usepackage{graphicx}
\usepackage{amsfonts}
\usepackage{amsmath}
\usepackage{dsfont}
\usepackage{bbm}
\usepackage{amssymb}
\usepackage{mathrsfs}
\usepackage[all]{xy}
\DeclareMathAlphabet{\mathpzc}{OT1}{pzc}{m}{it}

\DeclareMathOperator{\Z}{\mathbb{Z}}
\DeclareMathOperator{\bk}{\Bbbk}
\DeclareMathOperator{\R}{\mathbb{R}}
\DeclareMathOperator{\C}{\mathbb{C}}

\DeclareMathOperator{\G}{\mathbb{G}}

\DeclareMathOperator{\A}{\mathbb{A}}
\DeclareMathOperator{\K}{\mathbb{K}}
\newcommand{\CX}{\mathscr{X}}

\DeclareMathOperator{\pt}{\mathrm{pt}}

\DeclareMathOperator{\id}{\mathrm{id}}
\DeclareMathOperator{\pr}{\mathrm{pr}}
\DeclareMathOperator{\Hom}{\mathrm{Hom}}
\DeclareMathOperator{\sHom}{\mathpzc{Hom}}
\DeclareMathOperator{\Ext}{\mathrm{Ext}}
\DeclareMathOperator{\loc}{\mathrm{Loc}}
\DeclareMathOperator{\Rep}{\mathrm{Rep}}

\DeclareMathOperator{\str}{\mathcal{S}}
\DeclareMathOperator{\Vect}{\mathrm{Vect}}
\DeclareMathOperator{\Irr}{\mathrm{Irr}}

\DeclareMathOperator{\Aut}{\mathrm{Aut}}
\DeclareMathOperator{\End}{\mathrm{End}}

\DeclareMathOperator{\Sym}{\mathrm{Sym}}
\newcommand{\mi}{\mathfrak{m}}

\DeclareMathOperator{\car}{\mathrm{char}}

\DeclareMathOperator{\De}{\Delta}
\DeclareMathOperator{\Na}{\nabla}

\DeclareMathOperator{\fff}{\mathscr{F}}
\DeclareMathOperator{\fgg}{\mathscr{G}}
\DeclareMathOperator{\fqq}{\mathscr{Q}}
\DeclareMathOperator{\fhh}{\mathscr{H}}
\DeclareMathOperator{\fll}{\mathscr{L}}
\DeclareMathOperator{\fbk}{\underline{\Bbbk}}

\DeclareMathOperator{\For}{\mathrm{For}}
\DeclareMathOperator{\IC}{\mathrm{IC}}

\DeclareMathOperator{\Perv}{P_{\str}}

\newcommand{\Piso}{\mathrm{P}^t_{\mathrm{iso}}}
\newcommand{\Pco}{\mathrm{P}^t_{\mathrm{co}}}

\DeclareMathOperator{\indlim}{\varinjlim}

\DeclareMathOperator{\Cocar}{X_{\ast}}
\DeclareMathOperator{\Car}{X^{\ast}}

\DeclareMathOperator{\modu}{\mathrm{-mod}}

\DeclareMathOperator{\rak}{\mathrm{rk}}

\begin{document}
\numberwithin{equation}{section}
\newtheorem{thm}{Theorem}[section]
\newtheorem{lem}[thm]{Lemma}
\newtheorem{prop}[thm]{Proposition}
\newtheorem{cor}[thm]{Corollary}
\newtheorem{conj}[thm]{Conjecture}
\theoremstyle{definition}
\newtheorem{defn}[thm]{Definition}
\newtheorem{assump}[thm]{Assumption}
\newtheorem{conv}[thm]{Convention}
\theoremstyle{remark}
\newtheorem{rk}[thm]{Remark}
\newtheorem{ex}[thm]{Example}

\title{Perverse monodromic sheaves}
 
\author{Valentin Gouttard}
\address{Universit\'e Clermont Auvergne, CNRS, LMBP, F-63000 Clermont-Ferrand, France.}
\email{valentin.gouttard@uca.fr}

\begin{abstract}
    We introduce and study the category of modular (i.e.~with coefficient of positive characteristic) monodromic perverse sheaves on complex stratified $T$-varieties, with $T$ a complex algebraic torus. In particular, we show that under appropriate assumptions this category has a natural highest weight structure.
\end{abstract}

\maketitle


\section{Introduction}
\subsection{Main result}
Consider a complex algebraic torus $T$ and a stratified $T$-variety $(X, \str)$. We fix an algebraically closed (coefficient) field $\bk$, of positive characteristic $\ell >0$. We want to investigate in this paper the category $\mathrm{P}_{\str}(X, \bk)$ of perverse sheaves in $D^b_{\str}(X, \bk)$. 

The monodromy construction, due to Verdier (see \cite[\S 5]{Ve}) allows one to define an action of the group of cocharacters $\Cocar(T)$ of $T$ on each object of the category $D^b_{\str}(X, \bk)$. More precisely, we obtain an algebra morphism 
$$\varphi_{\fff} : \bk[\Cocar(T)]\rightarrow \End(\fff)$$
for any $\fff \in D^b_{\str}(X, \bk)$. It is a classical fact that the maximal ideals in $\bk[\Cocar(T)]$ are parametrized by the elements of the dual $\bk$-torus $T^\vee_{\bk}$. Considering such an element $t$ and the associated maximal ideal $\mi_t$, we can consider the full subcategory $D^b_{\str}(X, \bk)_{[\underline{t}]}$ of $D^b_{\str}(X, \bk)$ whose objects are those $\fff$ such that $\varphi_{\fff}$ factors through the quotient $\bk[\Cocar(T)]/\mi_t$. In particular, we obtain a full subcategory $\mathrm{P}_{\str}(X, \bk)_{[\underline{t}]}$ of $\mathrm{P}_{\str}(X, \bk)$. 
The inclusion of the closure of strata allows us to endow the set $\str$ with a poset structure $(\str, \leq)$.
The main result of this paper is then that, under
the hypothesis that each strata is isomorphic to a product of the form $\A^r \times T$ for some $r\in \Z_{\geq 0}$, the category $\mathrm{P}_{\str}(X, \bk)_{[\underline{t}]}$ is a highest weight category for any $t \in T^\vee_{\bk}$.

\subsection{Motivation}

Let $\mathfrak{g}$ be a complex semisimple Lie algebra.
In \cite{So}, Soergel proved two theorems (the Endomorphismensatz and the Struktursatz) giving a description of the principal block $\mathscr{O}_0$ of the BGG-category $\mathscr{O}$ of $\mathfrak{g}$ in terms of modules (the Soergel modules) over the endomorphism ring of the projective cover of the unique simple Verma module in $\mathscr{O}_0$. 
For a complex semisimple algebraic group $G$ with maximal torus $T$ and a Borel subgroup $B = T \ltimes U$ such that $\mathfrak{g}$ is the Lie algebra of $G$, Beilinson--Bernstein localization tells us that the category $\mathscr{O}_0$ for $\mathfrak{g}$ is equivalent to the category $\mathrm{P}_{(B)}(G/B, \C)$ of $B$-constructible complex perverse sheaves on the flag variety of $G$.

In the recent article \cite{BeR}, Bezrukavnikov and Riche obtained a totally geometric ``Soergel description" of $\mathrm{P}_{(B)}(G/B, \bk)$, for $\bk$ any field, i.e.~they described this category as modules over the endomorphism ring of the projective cover $P_e$ of the skyscraper sheaf at $B/B$. In order to obtain their result, they consider perverse objects on $G/U$. For any object in $D^b_{(B)}(G/U, \bk)$, we can define a monodromy morphism, coming from the natural right action of $T$ on $G/U$. We then obtain a family of subcategories in $\mathrm{P}_{(B)}(G/U, \bk)$ parametrized by the elements $t$ of the dual torus $T^\vee_{\bk}$; for $t = 1$, we end up with (a category canonically equivalent to) $\mathrm{P}_{(B)}(G/B, \bk)$. The case $t=1$ was the main point of interest of \cite{BeR}; in this paper, we begin to explore the categories given by other parameters $t$. We outline below the structure of the article

\subsection{Verdier's monodromy} In section \ref{section A result of Verdier}, we start by recalling Verdier's construction of the monodromy morphism for any object $\fff$ in our derived constructible category. In \S \ref{section preliminaries}, \ref{section Verdier's proposition} and \ref{section proof of verdier}, we elaborate on the constructions of \cite{Ve}. In \ref{sectiopn def }, we introduce the monodromy map and investigate some elementary properties. We then define the monodromic categories in \ref{subsection : Setting and Notations}.

\subsection{Lusztig and Yun monodromy} 
Section \ref{section Lusztig Yun} is devoted to the introduction of one of the key tools that allow us to prove our result, the equivariant monodromic category due to Lusztig--Yun.
We consider a slightly more general setting: $X$ is an $H$-variety, with $H$ any connected complex algebraic group.
Following the constructions of \cite{LY}, we define equivariant monodromic categories. We consider pairs $(\nu, \chi)$ where $\nu :  \widetilde{H}\rightarrow H$ is a finite central isogeny and $\chi$ a character of the kernel of $\nu$. Considering the $\chi$-isotypic component of the pushforward of the constant sheaf along $\nu$, we obtain a rank-one local system on $H$.
Such data allows us to define a ``character-equivariant" category, namely the Lusztig--Yun monodromic equivariant category. This is done in \S \ref{subsection preliminaries LY} and \S \ref{subsection equivariant monodromic LY}. One problem is that different pairs $(\nu, \chi)$ could lead to the same local system. The rest of section \ref{section Lusztig Yun} is devoted to dealing with this lack of uniqueness: under suitable assumptions on the characteristic of $\bk$, we show that a different pair leads to canonically equivalent categories.

\subsection{Perverse monodromic sheaves} In section \ref{section Perverse monodromic sheaves on stratified $T$-varieties}, we come back to the setting of stratified $T$-varieties. Let us fix a $t$ in $T^\vee_{\bk}$; we associate to this element a local system $\fll^T_t$ on $T$. The constructions of the preceding sections apply and we obtain two kinds of monodromic categories: a topological one (Verdier's construction) and an equivariant one (Lusztig--Yun construction). In both cases we can consider perverse objects; there is a major difference though: as explained in \cite{LY}, there is a natural perverse $t$-structure on the equivariant monodromic category, whose heart $\mathrm{P}^{\mathrm{LY}}_{ T,\str}(X, \bk)_{\fll^T_{t}}$ is (defined to be) the category of perverse monodromic equivariant sheaves. On the other side, our topological monodromic category $D^b_{\str}(X, \bk)_{[\underline{t}]}$ is not even a triangulated category, so the notion of $t$-structure does not make sense here. In particular the category $\mathrm{P}_{\str}(X, \bk)_{[\underline{t}]}$ is not really well behaved as a full subcategory of $D^b_{\str}(X, \bk)$, e.g.~it is not stable under extension. The next paragraph gives a way to solve this problem.
 
 \subsection{Equivalences} It is well known that the category of equivariant perverse sheaves admits several equivalent descriptions, as explained e.g. in \cite[\S A.1]{BaR}, in particular, a perverse object in the constructible derived category is equivariant if and only if its monodromy morphism factors through the quotient of $\bk[\Cocar(T)]$ by its natural augmentation ideal. Here, we consider perverse objects whose monodromy morphism factors through a quotient defined by another augmentation ideal, associated to some character of $\Cocar(T)$. In analogy with the usual equivariant case, we can give four different definitions of perverse monodromic sheaves. We show in \S 3.3 that these four definitions lead to equivalent categories; in particular, the categories $\Perv(X, \bk)_{[\underline{t}]}$ and $\mathrm{P}^{\mathrm{LY}}_{ T,\str}(X, \bk)_{\fll^T_{t}}$ are equivalent.
 
 \subsection{Highest-weight structure} In section \ref{section Highest weight structure of perverse monodromic sheaves}, we investigate the structure of the category of perverse monodromic sheaves under the assumption that each strata is isomorphic to a product of the form $\A^r \times T$ for some $r\in \Z_{\geq 0}$ (depending on the strata). This is an adaptation in the monodromic case of the main result of \cite[\S 3.2]{BGS}. There, the authors showed the following: the category of perverse sheaves on an algebraic variety stratified by linear affine spaces has a highest weight structure. 
 Our case is very similar. In this setting, we can introduce ``standard" and ``costandard" objects in analogy with \cite{BGS}. As stated above, the inclusion of closures of strata defines a partial order $\leq$ on $\str$,  and endows $\str$ with a poset structure. We show here that the category of perverse monodromic sheaves on $X$ together with the constructed standard and costandard objects and with the poset $(\str, \leq)$ form an highest weight category.
 
 \subsection{Further consideration}
 In an upcoming paper, we will investigate monodromic perverse sheaves on the quotient variety $G/U$ for $G$ a reductive group and $U$ the unipotent radical of some Borel subgroup. Generalizing the consideration of \cite{BeR}, we will give some ``Soergel description" of these categories.

\subsection{Acknowledgments} I am grateful to Zhiwei Yun for helpful insights about subsections \ref{section general consideration}-\ref{section fully faithfulness pnu}. I also thank deeply Pramod Achar and Simon Riche for endless patience and paying an extraordinary attention to this work, which could never have been done without their help.
\section{The topological monodromic category} \label{section A result of Verdier}
%
\subsection{Preliminaries} \label{section preliminaries}
In all this paper, we will assume that $\bk$ is the algebraic closure of a finite field. We let $\ell >0$ denote the characteristic of $\bk$.

We consider $A$ a complex algebraic torus of rank $r \geq 1$. We will use several times the fact that we can choose a trivialization $A \cong (\C^\ast)^r$ so we fix such a trivialization once and for all.

For any topological space $Y$, denote by $D^b(Y, \bk )$ the bounded derived category of sheaves of $\bk$-vector spaces on $Y$. Also, let $\loc(Y, \bk)$ be the abelian category of finite dimensional $\bk$-local systems on $Y$ (i.e. $\bk$-local systems on $Y$ with finite dimensional stalks).

Let $X$ be a complex algebraic variety endowed with an (algebraic) action of $A$. We fix a finite algebraic stratification $\str$ (see \cite[Definition 3.2.23]{CG}) such that each $S \in \str$ is $A$-stable.

Let $D^b_{\str}(X, \bk )$ denote the $\mathcal{S}$-constructible bounded derived category of sheaves of $\bk$-vector spaces on $X$.

We denote by $\pr_2 : A\times X \rightarrow X$ the projection and $a : A\times X \rightarrow X$ the action morphism. Let $e_n : A \rightarrow A$ be the morphism $t \mapsto t^n$. Finally, denote by $a(n)$ the composition $a \circ (e_n \times \id)$.

For a morphism $f$ between two algebraic varieties, the associated functors of pullback and (proper) pushforward will almost always be understood as derived functors. Thus, we will write $f_{\ast}$ instead of $Rf_{\ast}$ (and similarly for $f_!$ and $f^{\ast}$). In some places, we will have to consider regular, non-derived pushforward; the notations will then be $f^\circ_\ast$ and $f^\circ_!$ for the non-derived version.

\begin{lem}  \label{lemma automorphism finite order}
Let $V$ be a finite dimensional $\bk$-vector space. Then any $\phi \in GL_{\bk}(V)$ is of finite order.
\end{lem}
\begin{proof}
Fixing a basis in $V$, we obtain a group isomorphism $GL_{\bk}(V) \cong GL_{{\dim (V)}}(\bk)$. Now, since $\bk$ is the algebraic closure of a finite field, it is in particular the union of its finite subfields. This implies that $GL_{\dim (V)}(\bk)$ is the union of its finite subgroups. The result follows.
\end{proof}

We recall below a well-known theorem that we will use many times. A space $X$ is said semi-locally simply connected if each point $x \in X$ has a neighborhood $U_x$ such that any loop in $U_x$ is homotopic to a point when viewed as a path in $X$. Any complex algebraic variety is semi-locally simply connected.
\begin{thm} \label{theoreme representation systèmes locaux}
Let $Z$ be a connected, locally path connected, semi-locally simply connected topological space. Let $z_0 \in Z$ be any base-point. There is a canonical equivalence of categories 
$$\loc(Z,\bk) \cong \bk[\pi_1(Z, z_0)]\modu$$
where the right-hand side denotes the category of finite-dimensional $\bk$-represen\-tations of $\pi_1(Z, z_0)$.
Consider $(Y, y_0)$ another connected, locally path connected, semi-locally simply connected pointed topological space and $f : (Z, z_0) \rightarrow (Y,y_0)$ a continuous map of pointed spaces.
 The map $f$ induces a group morphism $$ \pi_1(f): \pi_1(Z,z_0) \rightarrow \pi_1(Y,y_0)$$ and the functor $f^\ast: \loc(Y,\bk) \rightarrow \loc(Z,\bk)$ corresponds to the restriction of scalars along this map at the level of representations: we have a commutative diagram 
 $$\xymatrix{\loc(Y,\bk)\ar[d]_{\wr} \ar[rr]^{f^\ast} & & \loc(Z,\bk)\ar[d]_{\wr} \\
\bk[\pi_1(Y, y_0)]\modu \ar[rr]_{Res_{\pi_1(f)} } & & \bk[\pi_1(Z, z_0)]\modu
}$$
\end{thm}

This theorem allows us to prove the following useful lemma. In this paper, unless stated otherwise, the fundamental group of any (topological) group will be considered with $1$ as base-point.

\begin{lem} \label{lemme tiré en arrière local devient constant}
Let $\mathscr{L}$ be a $\bk$-local system on $A$. Then there exists an integer $n$ such that $e_n^\ast(\mathscr{L})$ is a constant sheaf.
\end{lem}

\begin{proof}
Thanks to theorem \ref{theoreme representation systèmes locaux}, the local system $\mathscr{L}$ corresponds to a $\bk$-representation $V$ of $\pi_1(A)\cong \mathbb{Z}^r$. This representation is entirely determined by the action of the elements $f_i = (0,\ldots,0, \underset{i}{1}, 0, \ldots , 0) \in \mathbb{Z}^r$; this action is thus given by a family $\{g_i\}_{i=1,\ldots, r}$ of commuting elements in $GL_{\bk}(V)$. According to lemma \ref{lemma automorphism finite order}, the $g_i$'s are all of finite order; let $n$ the product of these orders.
For any $\bk$-representation $W$ of $\mathbb{Z}^r$ (equivalently, for any $\bk[x_i, x_i^{-1}\mid i = 1, \ldots, r]$-module) denote by $W_n$ the representation defined in the following way: as $\bk$-vector spaces, $W = W_n$, and the action ``$\;\cdot^n$" is given by 
$$x_i\cdot^n w = x_i^n \cdot w$$
for any $i$.
The functor $e_n^\ast : \loc(A, \bk) \rightarrow \loc(A, \bk)$ corresponds to the functor $\bk[\pi_1(Y, y_0)]\modu \rightarrow \bk[\pi_1(Y, y_0)]\modu, \;W \mapsto W_n$ via the equivalence of theorem \ref{theoreme representation systèmes locaux}.
Thus the action of $f_i$ on $V_n$ is given by the matrix $g_i^n = 1 \in GL_{\bk}(V)$. We get that $V_n$ is trivial and so $e_n^\ast(\mathscr{L})$ is a constant local system.
\end{proof}

\begin{rk}
Take $\mathscr{F}$ in $D^b_{\str}(X, \bk )$. By definition, the restriction to a stratum $S$ of any cohomology object of $\mathscr{F}$ is locally constant. By assumption, each $A$-orbit is contained in one of the strata ; each cohomology object of $\mathscr{F}$ is then locally constant on the $A$-orbits. This remark is in particular valid for a constructible sheaf.
\end{rk}

\begin{lem} \label{tiré en arrière d'un faisceaux devient constant sur les fibres}
Let $\mathscr{F}$ be an $\mathcal{S}$-constructible sheaf on $X$.
There exists an integer $n$ such that $a(n)^\ast(\mathscr{F})$ is constant on each fiber of $\pr_2$.
\end{lem}

\begin{proof}
Since $\mathcal{S}$ is finite, it suffices to show the lemma for $\mathscr{F}_{\mid S}$ for any $S \in \mathcal{S}$. 
We can therefore assume that $\mathscr{F}$ is a local system on $X$. Then $a^\ast(\mathscr{F})$ is a local system on $A \times X$ and thus corresponds to a $\bk$-representation $V$ of $\pi_1(A \times X,(1, x_0)) \cong \mathbb{Z}^r\times \pi_1(X,x_0)$ (for any base-point $x_0 \in X$). 
Thanks to theorem \ref{theoreme representation systèmes locaux}, the sheaf $a(n)^\ast(\fff)$ corresponds under the equivalence of theorem \ref{theoreme representation systèmes locaux} to the representation $\widetilde{V}$ of $\mathbb{Z}^r\times \pi_1(X,x_0)$, with $\widetilde{V}= V$ as $\bk$-vector space and $(m, [\alpha])$ acting as $(nm, [\alpha])$.
If we view $\widetilde{V}$ as a representation of $\mathbb{Z}^r$ (via the canonical injection $\mathbb{Z }^r\hookrightarrow \mathbb{Z}^r \times \pi_1(X, x_0)$), we get the restriction of scalar of the representation associated to $\fff$ along the group morphism $\pi_1(A) \rightarrow \pi_1(A)\times \pi_1(x_0, X)$ induced by the map 
$$A \rightarrow A\times\{x_0\} \hookrightarrow A \times X, \; t \longmapsto (t^n, x_0).$$
We deduce from lemma \ref{lemme tiré en arrière local devient constant} that for $n$ sufficiently large, this
corresponds to the trivial representation of $\mathbb{Z}^r$; the morphism defining the representation $ \widetilde{V}$ of $\mathbb{Z}^r \times \pi_1(X,x_0)$ hence factors through the projection to $\pi_1(X,x_0)$. This says exactly that  $a(n)^\ast(\mathscr{F})$ is of the form $\pr_2^\ast(\fgg)$ for $\mathscr{G}$ a local system on $X$.
\end{proof}

 \begin{lem} \label{tiré en arrière devient un pr2}
 
 Let $\mathscr{F}$ be a constructible sheaf on $A \times X$ that is constant on each fiber of $\pr_2$. Then there exists a sheaf $\mathscr{G}$ on $X$ such that 
 $$\mathscr{F}\cong \pr_2^\ast(\fgg).$$
 \end{lem}

\begin{proof}
Consider the object ${\pr_{2}}_!(\mathscr{F})$. For any $x \in X$, we have an isomorphism ${\pr_2}_!(\mathscr{F})_x \cong R\Gamma_c^{\bullet}(\mathscr{F}_{\mid A \times \{x\}})$ in the bounded derived category of $\bk$-vector spaces. Since $\mathscr{F}_{\mid A \times \{x\}}$ is constant by hypothesis, this cohomology lives in degrees $r, \ldots, 2r$.
We consider the non-zero truncation morphism ${\pr_2}_! (\fff) \rightarrow \fhh^{2r}({\pr_2}_!(\fff))[-2r]$. We use adjunction to obtain a (nonzero) morphism 
$$\fff \longrightarrow \pr_2^! \fhh^{2r}({\pr_2}_!(\fff))[-2r] \cong \pr_2^\ast\fhh^{2r}({\pr_2}_!(\fff)).$$
(We used here the fact that $\pr_2$ is a smooth morphism, so $\pr_2^! \cong \pr_2^\ast[2r]$.)

We check that this is an isomorphism by looking at the stalks. For a point $(z, x) \in A \times X$, we thus have a morphism 
$$\mathscr{F}_{(z,x)} \rightarrow \mathscr{H}^{2r}({\pr_2}_!(\mathscr{F}))_x \cong \mathbf{H}^{2r}_c(A\times \{x\}, \mathscr{F}).$$ We can then assume that $X$ is a one-point space and hence that $\mathscr{F}$ is a constant sheaf. It therefore suffices to check the isomorphism for $\mathscr{F}$ constant of rank 1; in this case, both sides are just $\bk$ (thanks to K\"unneth's formula applied to $A \cong (\C^\ast)^r$). Since our morphism is nonzero, there exists a point $(z,x)$ such that the stalk at $(z,x)$ is nonzero. Since $\fff$ is assumed to be constant, the (global) morphism is nonzero.
\end{proof}

\subsection{Verdier's proposition} \label{section Verdier's proposition}
%

\begin{prop} \label{proposition Verdier 5.1}
Let $\mathscr{F}$ be an object in $D^{b}_{\str}(X, \bk )$. There exists $n \in \mathbb{Z}_{>0}$ and a morphism 
$$\iota(n) : \pr_2^\ast(\mathscr{F}) \longrightarrow a(n)^\ast (\mathscr{F})  $$
such that $\iota_{\mid\{1\}\times X}$ identifies with the identity of $\mathscr{F}$. Any such morphism is an isomorphism.
If $n_1$, $n_2$ are two strictly positive integers and if 
$$\iota(n_1) : \pr_2^\ast(\mathscr{F}) \longrightarrow a(n_1)^\ast (\mathscr{F}), \quad  \iota(n_2) : \pr_2^\ast(\mathscr{F}) \longrightarrow a(n_2)^\ast (\mathscr{F})$$
are two such morphisms, then there exists a strictly positive integer $n_3$, multiple of both $n_1$ and $n_2$, such that 
$$(e_{\frac{n_3}{n_1}}\times \id)^\ast(\iota(n_1))= (e_{\frac{n_3}{n_2}}\times \id)^\ast(\iota(n_2)).$$
Finally, let $\mathscr{G}$ be another complex in $D^{b}_{\str}(X, \bk )$ and $u : \mathscr{F}\rightarrow \mathscr{G}$ be any morphism. There exists $n \in \Z_{>0}$ and two morphisms $\iota_1(n): \pr_2^\ast(\mathscr{F})\rightarrow a(n)^\ast (\mathscr{F}) $ and $\iota_2(n): \pr_2^\ast(\mathscr{G})\longrightarrow a(n)^\ast (\mathscr{G})$ as above such that the following diagram commutes: 
\begin{equation}\label{diagramme verdier proposition 5.1}\vcenter{\xymatrix{
\pr_2^\ast(\mathscr{F})\ar[d]_{\pr_2^\ast(u)} \ar[rr]^{\iota_1(n)}_\sim && a(n)^\ast(\mathscr{F}) \ar[d]^{a(n)^\ast(u)} \\
\pr_2^\ast(\mathscr{G}) \ar[rr]^{\iota_2(n)}_\sim && a(n)^\ast(\mathscr{G}).
}}
\end{equation}
\end{prop}

Before giving the proof of the proposition, we need an intermediary result. We have 
$$(e_{n}\times \id)^\ast \pr_2^\ast = \pr_2^\ast,$$
so the inverse image functor of $(e_{n}\times \id)$ induces an endomorphism 
\begin{equation*}
\begin{aligned}\rho_n :  \Hom_{D^b_{\str}(A\times X, \bk)}(\pr_2^\ast(\mathscr{F}), \pr_2^\ast(\mathscr{G})) \longrightarrow  \Hom_{D^b_{\str}(A\times X, \bk)}(\pr_2^\ast(\mathscr{F}), \pr_2^\ast(\mathscr{G})).
\end{aligned}
\end{equation*}
(Here, with a slight abuse of notation, we denote again by $\str$ the stratification of $A \times X$ whose strata are the $A \times S$ for $S\in \str$.)
We clearly have $\rho_n \circ \rho_m = \rho_{nm}$. Moreover, the following diagram commutes: 

$$\xymatrix{
\Hom_{D^b_{\str}( X, \bk)}(\mathscr{F}, \mathscr{G}) \ar[d]_{\pr_2^\ast} \ar[rd]^{\pr_2^\ast} & \\
\Hom_{D^b_{\str}(A\times X, \bk)}(\pr_2^\ast(\mathscr{F}), \pr_2^\ast(\mathscr{G})) \ar[r]_{\rho_n} & \Hom_{D^b_{\str}(A\times X, \bk)}(\pr_2^\ast(\mathscr{F}), \pr_2^\ast(\mathscr{G})). 
}$$
We define an inductive system, parametrized by the (filtrant) set $\Z_{>0}$ endowed with the divisibility relation. For any $n > 0$, we set
 $$V_n = \Hom_{D^b_{\str}(A\times X, \bk)}(\pr_2^\ast(\mathscr{F}), \pr_2^\ast(\mathscr{G})).$$ 
If the integer $n$ divides $m$ the morphism $V_n \rightarrow V_m$ is given by $\rho_{\frac{m}{n}}$. 
The above diagram states that $\pr_2^\ast$ defines a morphism from the object $\Hom(\mathscr{F}, \mathscr{G})$ to the limit $\indlim_n V_n$. 

\begin{lem} \label{lemme isomorphisme de Ind sur les Hom}

We have an isomorphism of vector spaces induced by $\pr_2^\ast$
\begin{equation} \label{equation isomorphisme ind-objet}
    \Hom_{D^b_{\str}(X, \bk)}(\mathscr{F}, \mathscr{G}) \overset{\sim}{\longrightarrow} \varinjlim_n \Hom_{D^b_{\str}(A\times X, \bk)}(\pr_2^\ast(\mathscr{F}), \pr_2^\ast(\mathscr{G})).
\end{equation}
\end{lem}

\begin{proof}
First, note that $\pr_2$ is a smooth morphism since $A$ is smooth. Thus we have $\pr_2^! \cong \pr_2^\ast[2\dim(A)]=\pr_2^\ast[2r]$.
Now, thanks to \cite[(2.6.4)]{KS}, we can write 
$$R\Hom(\pr_2^\ast(\mathscr{F}), \pr_2^\ast(\mathscr{G})) = R\Gamma(A \times X, R\sHom(\pr_2^\ast(\mathscr{F}), \pr_2^\ast(\mathscr{G})).$$
Using \cite[Proposition 3.1.13]{KS}, we get the following isomorphisms in the category $D^b_{\str}(A\times X, \bk)$: 
\begin{align*}
    \pr_2^\ast \left(R\sHom(\mathscr{F}, \mathscr{G}) \right) &\cong \pr_2^! \left(R\sHom(\mathscr{F}, \mathscr{G})\right)[-2r] \\
    & \cong R\sHom(\pr_2^\ast(\mathscr{F}), \pr_2^!(\mathscr{G}))[-2r] \\
    & \cong R\sHom\left(\pr_2^\ast(\mathscr{F}), \pr_2^\ast(\mathscr{G})[2r]\right)[-2r] \\
    & \cong R\sHom(\pr_2^\ast(\mathscr{F}), \pr_2^\ast(\mathscr{G})).
\end{align*}
Using K\"unneth's formula, we obtain a chain of isomorphisms, in the derived category $D^b\Vect_{\bk}$ of $\bk$-vector spaces: 
\begin{align*}
     R\Hom(\pr_2^\ast(\mathscr{F}), \pr_2^\ast(\mathscr{G})) &\cong R\Gamma(A \times X, R\sHom(\pr_2^\ast(\mathscr{F}), \pr_2^\ast(\mathscr{G})) \\
     &\cong R\Gamma(A \times X, \pr_2^\ast R\sHom(\mathscr{F}, \mathscr{G}))\\
     &\cong R\Gamma(A \times X, \bk \boxtimes R\sHom(\mathscr{F}, \mathscr{G}))\\
     &\cong R\Gamma(A, \bk) \otimes^L_{\bk} R\Hom(\mathscr{F}, \mathscr{G}).
\end{align*}
Since we work over a field $\bk$, we know that the tensor product functor over $\bk$ is an exact functor, so taking cohomology on both sides, we get an isomorphism of vector spaces
$$\Hom_{D^b_{\str}(A\times X, \bk)}(\pr_2^\ast(\mathscr{F}), \pr_2^\ast(\mathscr{G})) \cong \bigoplus_{i\in \Z}\mathbf{H}^i(A, \bk) \otimes_{\bk} \Hom_{D^b_{\str}(X, \bk)}(\mathscr{F}, \mathscr{G}[-i])$$
where $\mathbf{H}^i(A, \bk)$ is the $i$-th cohomology of $A$.
Under these isomorphisms, the transition morphism $\rho_n$ identifies with $\rho_n(\bk)\otimes \id$, where $\rho_n(\bk)$ denote the morphism $\mathbf{H}^\bullet(A, \bk)\rightarrow \mathbf{H}^\bullet(A, \bk)$ induced by $e_n^\ast$.


As inductive limits commute with tensor products, if we show that the inductive system $(V'_n : = \mathbf{H}^\bullet(A, \bk), \rho_{\frac{m}{n}}(\bk))$ satisfies
\begin{equation}\label{limite cohomologie du tore}\indlim_n \mathbf{H}^i(A, \bk) = \left\{
\begin{array}{l}
  \bk \mbox{  if } i=0 \\
 0\; \mbox{   otherwise}
\end{array}
\right.
\end{equation}
we can conclude.
Assume that $A$ is of rank 1, i.e. that $A \cong \mathbb{G}_m=\C^\ast$.
Recall that the cohomology of $\mathbb{G}_m$ is non-zero only in degree 0 and in degree 1.
Note that we have a morphism 
$$\bk \longrightarrow \indlim\mathbf{H}^\bullet(\mathbb{G}_m, \bk)$$
defined by sending the $\bk$ on the left onto the copy of $\bk$ in degree zero of $V'_1 =\mathbf{H}^\bullet(\mathbb{G}_m, \bk)$. 

We know that $\mathbf{H}^\bullet(\mathbb{G}_m, \bk) = \Ext^\bullet_{D^b(\G_m,\bk)}(\underline{\bk}_{\mathbb{G}_m}, \underline{\bk}_{\mathbb{G}_m})$. Using the (easy) fact that an extension between two local systems on $\G_m$ is itself a local system,  we have 
\begin{align*}
    \mathbf{H}^\bullet(\mathbb{G}_m, \bk) & = \mathbf{H}^0(\G_m, \bk) \oplus\mathbf{H}^1(\G_m, \bk) \\
    & \cong \Hom_{D^b(\G_m,\bk)}(\underline{\bk}_{\mathbb{G}_m}, \underline{\bk}_{\mathbb{G}_m})\oplus \Ext^1_{D^b(\G_m,\bk)}(\underline{\bk}_{\mathbb{G}_m}, \underline{\bk}_{\mathbb{G}_m}) \\
    &= \Hom_{\loc(\mathbb{G}_m, \bk)}(\underline{\bk}_{\mathbb{G}_m}, \underline{\bk}_{\mathbb{G}_m})\oplus \Ext_{\loc(\mathbb{G}_m, \bk)}^1(\underline{\bk}_{\mathbb{G}_m}, \underline{\bk}_{\mathbb{G}_m})\\
    & \cong \Hom_{\Rep(\mathbb{Z}, \bk)}(\bk_{\mathrm{triv}}, \bk_{\mathrm{triv}})\oplus \Ext_{\Rep(\mathbb{Z}, \bk)}^1(\bk_{\mathrm{triv}}, \bk_{\mathrm{triv}}).
\end{align*}
We can thus work in the category of $\bk$-representations of $\mathbb{Z}$.
We have to determine the action of $\rho_n(\bk)$ on each of these summands.

 It is obvious that $\bk_{\mathrm{triv},n} = \bk_{\mathrm{triv}}$ (see the proof of the lemma \ref{lemme tiré en arrière local devient constant} for the notation), so the action of $\rho_n(\bk)$ on $\Ext_{\Rep(\mathbb{Z}, \bk)}^0(\bk_{\mathrm{triv}}, \bk_{\mathrm{triv}})= \Hom_{\Rep(\mathbb{Z}, \bk)}(\bk_{\mathrm{triv}}, \bk_{\mathrm{triv}})$ is simply the identity.
Now, we know that $\Ext^1_{\Rep(\mathbb{Z}, \bk)}(\bk_{\mathrm{triv}}, \bk_{\mathrm{triv}})$ is one-dimensional; an isomorphism $\bk \longrightarrow \Ext^1_{\Rep(\mathbb{Z}, \bk)}(\bk_{\mathrm{triv}}, \bk_{\mathrm{triv}})$ is given by 
$$ x \longmapsto V(x)$$
where $V(x) \cong \bk \oplus \bk$ as a vector space and the action of $1 \in \mathbb{Z}$ is given by the matrix $\begin{pmatrix} 
1 & x \\
0 & 1
\end{pmatrix}$. The action of $1 \in \mathbb{Z}$ on $V(x)_n$ is then given by the matrix $\begin{pmatrix} 
1 & x \\
0 & 1
\end{pmatrix}^n = \begin{pmatrix} 
1 & nx \\
0 & 1
\end{pmatrix}$. The endomorphism of $\bk$ induced by $\rho_n(\bk)$ is thus multiplication by $n$. We have determined the action of $\rho_n(\bk)$ on each of the summands of $\mathbf{H}^\bullet(\G_m, \bk)$.
As limit, we obtain 
$\bk \oplus\varinjlim \bk.$
In the second summand, the transition map $V'_n \supseteq \bk \rightarrow \bk \subseteq V'_m$ is the multiplication by $\frac{m}{n}$ (for $m$ divisible by $n$).
 Since $\bk$ is the algebraic closure of a finite field of characteristic $\ell$, multiplication by $n$ is zero as soon as $\ell$ divides $n$ ; this limit is then zero and \eqref{limite cohomologie du tore} holds.

We now deal with the general case: we have a commutative diagram 
\begin{equation}\label{diagram lemme trivialisation de en}\vcenter{\xymatrix{
A \ar[rr]^{e_n}\ar[d]^{\wr} && A \ar[d]^{\wr}\\
(\C^\ast)^r \ar[rr]_{(e_n \times \cdots \times e_n)} && (\C^\ast)^r
}}\end{equation}
(we made a little abuse of notation, denoting by the same symbol ``$e_n$" the maps $A \rightarrow A$ and $\C^\ast \rightarrow \C^\ast$ sending an element on its $n$-th power).
Using K\"unneth's formula one more time, we get 
\begin{equation}\label{formule de kunneth pour T}\mathbf{H}^m(A, \bk) \cong \bigoplus_{p_1 + \cdots + p_r = m}\left(\bigotimes_{i=1}^r\mathbf{H}^{p_i}(\C^\ast, \bk)\right).
\end{equation}
Thanks to the diagram \eqref{diagram lemme trivialisation de en}, the induced map $\rho_n(\bk) : \mathbf{H}^n(A, \bk) \rightarrow \mathbf{H}^n(A, \bk)$ decomposes in a direct sum of tensor product of maps, each of which corresponding to the map already studied in the case $A = \G_m$. Thanks to \eqref{formule de kunneth pour T}, the cohomology of $A$ is in degrees $0, \ldots, r$. From the case $A = \G_m$, we deduce immediately that the induced map $\rho_n(\bk)$ is the identity on $\mathbf{H}^0(A, \bk) = \bk$.
Now for $\mathbf{H}^m(A, \bk)$ with $m\geq 1$, there is in each summand at least one $p_i = 1$. On this summand, the map induced by $\rho_n(\bk)$ is zero if $n$ is divisible by $\ell$. Passing to the inductive limit, we deduce that 
\eqref{limite cohomologie du tore} holds again in this case.
We can conclude as above and the proof is complete.
\end{proof}
\bigbreak 
\subsection{Proof of Verdier's proposition}
\label{section proof of verdier}

\begin{proof}[Proof of proposition \ref{proposition Verdier 5.1}] 
Let $\mathscr{F}\in D_{\str}^b(X, \bk)$. We prove the existence of an isomorphism between $a(n)^\ast(\fff)$ and an object of the form $\pr^\ast_2(\fqq)$ with $\fqq \in D^b_{\str}(X, \bk)$ (for some $n \in \Z_{>0}$) by induction on the minimal length of an interval containing $\{i \in \mathbb{Z} \mid \mathscr{H}^i(\mathscr{F})\neq 0\}$. 

Assume that $\mathscr{F}$ has a unique non-zero cohomology object. We can use lemma \ref{tiré en arrière d'un faisceaux devient constant sur les fibres} and lemma \ref{tiré en arrière devient un pr2}. We get an integer $n$ such that $a(n)^\ast(\fff)$ is constant on the fibers of $\pr_2$.
Since $a(n)^\ast$ is an exact functor, we obtain an isomorphism $a(n)^\ast(\fff) \cong \pr^\ast_2(\fgg)$ for a certain sheaf $\fgg$ on $X$. Restricting to $\{1\}\times X$, we get $\fff \cong \fgg$. This settles this case.
Now for a general $\mathscr{F}$, denote by $N$ the largest integer such that $\mathscr{H}^N(\mathscr{F})\neq 0$. We have a truncation triangle 
$$\mathscr{H}^N(\mathscr{F})[-N-1] \overset{f}{\longrightarrow}\tau_{\leq N-1}\mathscr{F} \longrightarrow \mathscr{F} \overset{[1]}{\longrightarrow} .$$
Denote $\mathscr{G}:=\tau_{\leq N-1}\mathscr{F}$ and $\mathscr{K}:= \mathscr{H}^N(\mathscr{F})[-N-1]$. Both $\mathscr{G}$ and $\mathscr{K}$ have a smaller amplitude than $\mathscr{F}$ so by induction, there exists an $n$ such that $a(n)^\ast(\mathscr{G})$ and $a(n)^\ast(\mathscr{K})$ are of the form $\pr_2^\ast(\mathscr{V})$ and $\pr_2^\ast(\mathscr{U})$.

The morphism $a(n)^\ast(f)$ defines a morphism $\pr_2^\ast(\mathscr{U})\rightarrow\pr_2^\ast(\mathscr{V})$, and thus an element in $\indlim \Hom(\pr_2^\ast(\mathscr{U}), \pr_2^\ast(\mathscr{V}))$ (say in the copy of $\Hom(\pr_2^\ast(\mathscr{U}), \pr_2^\ast(\mathscr{V}))$ indexed by 1 in the inductive limit). With the isomorphism \eqref{equation isomorphisme ind-objet}, replacing $n$ by a multiple if necessary, the morphism $a(n)^\ast(f)$ is of the form $\pr_2^\ast(g)$ for some $g : \mathscr{U} \rightarrow \mathscr{V}$; hence $a(n)^\ast(\mathscr{F})$ identifies with the cone of $\pr_2^\ast(g)$ and is of the form $\pr_2^\ast(\mathscr{Q})$.

So far, we have obtained an isomorphism 
$$ \pr_2^\ast(\mathscr{Q})\overset{\sim}{\longrightarrow} a(n)^\ast(\mathscr{F})$$
for a certain integer $n$ and an object $\mathscr{Q}\in D^b_{\str}(X, \bk)$. To identify $\mathscr{Q}$, we can restrict to the subset $\{1\} \times X$ to see that $\mathscr{F}\cong \mathscr{Q}$. We have thus obtained an isomorphism 
$$\iota(n) : \pr_2^\ast(\mathscr{F}) \overset{\sim}{\longrightarrow} a(n)^\ast(\mathscr{F}) .$$
It is also clear that $\iota(n)$ restricts to the identity on $\{1\}\times X$.

We now prove the second statement. We know that there exists an isomorphism $\iota(n) : \pr_2^\ast(\fff) \rightarrow a(n)^\ast(\fff)$ whose restriction to $\{1\}\times X$ is the identity of $\fff$. Consider a morphism $\iota : \pr_2^\ast(\fff) \rightarrow a(n)^\ast(\fff) $ such that $\iota_{\mid \{1\}\times X} = \id_{\fff}$ (to begin with, for the same $n$).
For $m$ divisible by $n$, the composition 
$$ (e_{\frac{m}{n}}\times \id)^\ast\iota \circ (e_{\frac{m}{n}}\times \id)^\ast\iota(n)^{-1}$$
 is an endomorphism of $\pr_2^\ast(\mathscr{F})$. It then defines an element in 
 $$\indlim \Hom(\pr_2^\ast(\mathscr{F}), \pr_2^\ast(\mathscr{F})).$$
 (Once again, we choose our element in the first copy of $\Hom(\pr_2^\ast(\mathscr{F}), \pr_2^\ast(\mathscr{F}))$ appearing in the limit.)
 With the isomorphism \eqref{equation isomorphisme ind-objet}, and replacing $m$ by a multiple if necessary, we can assume that this element is of the form $\pr_2^\ast(f)$ for $f$ an endomorphism of $\mathscr{F}$. Restricting again to $\{1\}\times X$, we obtain $f = \id_{\mathscr{F}}$. For $m$ sufficiently large, we deduce that $(e_{\frac{m}{n}}\times \id)^\ast(\iota\circ \iota(n)^{-1}) =\id_{\pr_2^\ast\fff}$; then 
 \begin{equation}\label{iota est un isomorphisme}(e_{\frac{m}{n}}\times \id)^\ast\iota = (e_{\frac{m}{n}}\times \id)^\ast\iota(n).\end{equation}
 To check that a morphism between complexes of sheaves is an isomorphism in the derived category, it suffices to do so on the stalks of the cohomology objects (or the cohomology objects of the stalks). With the equality \eqref{iota est un isomorphisme} and noticing that $e_{\frac{m}{n}}$ is a surjective map, we get that $\iota$ is an isomorphism on the stalks and hence an isomorphism.

Now, if $n_1$ and $n_2$ are as in the proposition and if $n_3$ is a common multiple of $n_1$ and $n_2$, we note that (with the notations of the proposition) $(e_{\frac{n_3}{n_1}}\times \id)^\ast\iota(n_1)$ and $(e_{\frac{n_3}{n_2}}\times \id)^\ast\iota(n_2)$  are morphisms $ \pr_2^\ast(\mathscr{F})\rightarrow a(n_3)^\ast(\mathscr{F})$ whose restriction to $\{1\}\times X$ is again the identity of $\mathscr{F}$, so we may as well assume that $n_1=n_2$ and hence we can use the above reasoning to conclude.

Finally, consider a second complex $\mathscr{G}$ and a morphism $ u : \mathscr{F}\rightarrow \mathscr{G}$. Choose $n$ large enough to have two isomorphisms 
$$\iota_1(n) : \pr_2^\ast(\mathscr{F}) \longrightarrow a(n)^\ast(\mathscr{F}),\qquad \iota_2(n) : \pr_2^\ast(\mathscr{G}) \longrightarrow a(n)^\ast(\mathscr{G}). $$
The composition 
$\iota_2(n)^{-1}\circ a(n)^\ast(u) \circ \iota_1(n)$ is a morphism $\pr_2^\ast(\mathscr{F}) \rightarrow \pr_2^\ast(\mathscr{G})$. Using again the isomorphism \eqref{equation isomorphisme ind-objet} and replacing $n$ by a multiple if necessary, we can assume that this morphism is of the form $\pr_2^\ast(v)$ for $v : \mathscr{F}\rightarrow \mathscr{G}$. Restricting to $\{1\}\times X$, we obtain $v = u$ which finishes the proof.
\end{proof}

%
\subsection{Definition}\label{sectiopn def }
%
We define here the monodromy of a constructible complex. 
For any cocharacter $\lambda\in \Cocar(A)$ and any integer $n$, we set $\lambda(e^{\frac{2i\pi}{n}}) = \lambda_n$.
Let also
$$ j_{\sharp} : \{\sharp\}\times X \hookrightarrow A\times X\quad \mbox{and}\quad \tau_{\sharp} : X \overset{\sim}{\rightarrow} \{\sharp\}\times X$$
denote the inclusion and 
the canonical map respectively, for any $\sharp\in A$.

Consider an object $\fff \in D^b_{\str}(X, \bk)$ and a cocharacter $\lambda \in \Cocar(A)$.
Thanks to proposition \ref{proposition Verdier 5.1}, there exists $n \in \Z_{>0}$ and an isomorphism 
$$\iota(n) : \pr_2^\ast(\fff) \rightarrow a(n)^{\ast}(\fff)$$
whose restriction to $\{1\}\times X$ identifies with the identity of $\fff$.
We set
$$\varphi_{\fff,n}^\lambda : =  \tau_{\lambda_n}^\ast j^\ast_{\lambda_n}\left(\iota(n)\right): \fff \rightarrow \fff.$$
Our first task is to show that in fact the isomorphism $\varphi_{\fff,n}^\lambda$ does not depend on the integer $n$.

\begin{prop} \label{proposition monodrmie bien définie}
Consider two integers $n_1,n_2$ such that there exist two morphisms   $$\iota_1 : \pr_2^\ast(\fff) \rightarrow a(n_1)^{\ast}(\fff) \quad \mbox{and} \quad \iota_2 : \pr_2^\ast(\fff) \rightarrow a(n_2)^{\ast}(\fff)$$
as in proposition \ref{proposition Verdier 5.1}. Then for any $\lambda \in \Cocar(A)$, we have 
$$ \varphi_{\fff,n_1}^\lambda =\varphi_{\fff,n_2}^\lambda.$$
\end{prop}

\begin{proof} We want to show that 
$$ \tau_{\lambda_{n_1}}^\ast j_{\lambda_{n_1}}^\ast(\iota_1) = \tau_{\lambda_{n_2}}^\ast j_{\lambda_{n_2}}^\ast(\iota_2).$$
We know that there exists a multiple $m$ of both $n_1$ and $n_2$ such that $(e_{\frac{m}{n_1}}\times \id)^\ast(\iota) = (e_{\frac{m}{n_2}}\times \id)^\ast(\iota')$. Thus 
$$j_{\lambda_m}^\ast(e_{\frac{m}{n_1}}\times \id)^\ast(\iota_1) = j_{\lambda_m}^\ast(e_{\frac{m}{n_2}}\times \id)^\ast(\iota_2).$$

It is then easy to deduce that $\tau_{\lambda_{n_1}}^\ast j_{\lambda_{n_1}}^\ast(\iota_1)$ and $\tau_{\lambda_{n_2}}^\ast j_{\lambda_{n_2}}^\ast(\iota_2)$ are already equal in $ \Hom(\fff, \fff)$.
\end{proof}
Thanks to proposition \ref{proposition monodrmie bien définie}, we now can get rid of the $n$ in the notation $\varphi_{\fff,n}^\lambda$.
For any $\fff \in D^b_{\str}(X, \bk)$ have obtained a map 
$$\varphi_{\fff} : \Cocar(A) \rightarrow \Aut_{D^b_{\str}(X, \bk)}(\fff), \quad \lambda \mapsto \varphi_{\fff}^\lambda.$$

We state now a useful lemma.
\begin{lem} \label{monodromie est un morphisme de groupe}
Fix an object $\fff \in D^b_{\str}(X, \bk)$. The map $\varphi_{\fff}$ is a group morphism $ \Cocar(A) \rightarrow \Aut_{D^b_{\str}(X, \bk)}(\fff)$.
\end{lem}

\begin{proof}
Consider $\lambda, \mu\in \Cocar(A)$. Consider an integer $n$ such that there exists a morphism 
$$\iota : \pr_2^\ast \fff \rightarrow a(n)^\ast\fff$$
whose restriction to $\{1\}\times X$ identifies with $\id_{\fff}$.

Let $f : A\times X \rightarrow A \times X$ denote the map 
$(t,x) \longmapsto (\mu_n t,x).$
As we have $a(n)\circ f = a(n)$ and $\pr_2 \circ f = \pr_2$, the functor $f^\ast$ induces a map 
$$\Hom_{D^b_{\str}(A\times X)}(\pr_2^\ast\fff ,a(n)^\ast \fff ) \rightarrow \Hom_{D^b_{\str}(A\times X)}(\pr_2^\ast\fff,a(n)^\ast \fff ).$$
We then consider the composition 
\begin{equation} \label{morphism monodromie groupe}
\xymatrix{\pr_2^\ast\fff \ar[rr]^{\pr_2^\ast((\varphi^\mu_{\fff})^{-1})} && \pr_2^\ast\fff \ar[rr]^{f^\ast(\iota)} &&  a(n)^\ast \fff.
}\end{equation}
One can check that this morphism restricts to $\id_{\fff}$ on $\{1\}\times X$. 
Thanks to proposition \ref{proposition monodrmie bien définie}, we can define $\varphi_{\fff}^\lambda$ by restricting \eqref{morphism monodromie groupe} to $\{\lambda_n\}\times X$ (and then applying $\tau_{\lambda_n}^\ast$).
Doing so, we obtain $\tau_{\lambda_n}^\ast j_{\lambda_n}^\ast(f^\ast(\iota)) \circ (\varphi^\mu_{\fff})^{-1}$.
To conclude the proof, we must show that $\tau_{\lambda_n}^\ast j_{\lambda_n}^\ast(f^\ast(\iota)) = \varphi_{\fff}^{\lambda\mu}$. But this is clear from the following commutative diagram: 
$$\xymatrix{
X \ar[rr]_{\sim\qquad}^{\tau_{\lambda_n}\qquad}\ar@{=}[d] && \{\lambda_n\}\times X \ar[d]_{f_{\mid \{\lambda_n\}\times X}} \ar@{^{(}->}[rr]^{\qquad j_{\lambda_n}} && A \times X \ar[d]_{f} \\
X \ar[rr]_{\sim\qquad}^{\tau_{\lambda_n \mu_n}\qquad} &&\{\lambda_n\mu_n\}\times X \ar@{^{(}->}[rr]^{\qquad j_{\lambda_n\mu_n}} && A \times X.
}
$$
The proof is now complete.
\end{proof}
Thanks to lemma \ref{monodromie est un morphisme de groupe}, for any $\fff$, we can extend the morphism $\varphi_{\fff} $ to a $\bk$-algebra morphism:
$$\varphi_{\fff} : \bk[\Cocar(A)] \longrightarrow \End_{D^b_{\str}(X, \bk)}(\fff).$$
\begin{defn} \label{definition monodromy}
The morphism of algebras $\varphi_{\fff} : \bk[\Cocar(A)] \longrightarrow \End_{D^b_{\str}(X, \bk)}(\fff)$ is called the canonical monodromy morphism of $\fff$. The action of $\Cocar(A)$ on $\fff$ defined by $\varphi_{\fff}$ is the canonical monodromy action on $\fff$.  
\end{defn}

The following result is a key feature of the monodromy morphism:

\begin{lem} \label{lemme monodromie commute}
Consider any $\lambda \in \Cocar(A)$. For any $\fff, \fgg$ in $D^b_{\str}(X, \bk)$ and any morphism $f : \fff \rightarrow \fgg$, we have $f \circ \varphi_{\fff}^\lambda = \varphi_{\fgg}^\lambda \circ f$.
\end{lem}

\begin{proof}
This is an immediate consequence of the commutativity of diagram \eqref{diagramme verdier proposition 5.1} of proposition \ref{proposition Verdier 5.1}: with the notations therein, the lemma follows from applying the functor of restriction to $\{\lambda(e^{\frac{2i\pi}{n}})\}\times X$ to the diagram. 
\end{proof}

\begin{lem} \label{lemme pullback monodromy}
Consider a locally closed union of strata $Z \subseteq X$ and let $j : Z \hookrightarrow X$ be the inclusion map. We denote by $\str_Z$ the induced stratification on $Z$. Then for any $\fff \in D^b_{\str_Z}(Z, \bk)$ and any $\fgg\in D^b_{\str}(X, \bk)$, we have: \begin{enumerate}
    \item $\varphi^{\lambda}_{j_!(\fff)}= j_!(\varphi^{\lambda}_{\fff})$, 
    \item $\varphi^{\lambda}_{j_{\ast}(\fff)}= j_\ast(\varphi^{\lambda}_{\fff})$, 
    \item $\varphi^{\lambda}_{j^{\ast}(\fgg)}= j^\ast(\varphi^{\lambda}_{\fgg})$
    for any $\lambda \in \Cocar(A)$.
\end{enumerate} 

\end{lem}
The proof follows easily from the definition of the monodromy morphism and applications of the (smooth) base change theorem.

%
\subsection{Monodromic categories}  \label{subsection : Setting and Notations}
%

Denote by $\gamma$ the map $[0,1] \rightarrow \C^\ast$ which maps $t$ to $e^{2i\pi t}$.
It is a classical fact that the map $\lambda \mapsto \lambda \circ \gamma$ gives a group isomorphism $\Cocar(A) \rightarrow \pi_1(A)$. Using theorem \ref{theoreme representation systèmes locaux}, we obtain an equivalence $\loc(A, \bk) \cong \bk[\Cocar(A)]\modu$, where the latter category is the category of finite dimensional $\bk[\Cocar(A)]$-modules (or equivalently, finite dimensional $\bk$-representation of the group $\Cocar(A)$).
For any $\lambda \in \Cocar(A)$, we denote the corresponding basis element of $\bk[\Cocar(A)]$ by $e^\lambda$. 
We consider the dual $\bk$-torus $A^\vee_{\bk}$ of $A$. By definition, this is the algebraic torus over $\bk$ whose character lattice is the cocharacter lattice of $A$ i.e. $\Cocar(A)= \Car(A^{\vee}_{\bk})$. It is a standard fact that the algebra of regular functions on a torus is isomorphic to the group algebra of the characters of the torus. We thus get
$$\bk[\Cocar(A)] = \bk[\Car(A^{\vee}_{\bk})] \cong \mathscr{O}(A^{\vee}_{\bk}).$$
The irreducible $\bk$-representations of $\Cocar(A)$ are one-dimensional since this group is abelian. These representations correspond to group morphisms
$\Cocar(A)\rightarrow \bk^\ast$
which in turn correspond to $\bk$-algebra morphisms $\bk[\Cocar(A)]  \rightarrow \bk$. 
Using the above isomorphisms, such morphisms are the same as algebra morphisms 
$\mathscr{O}(A^{\vee}_{\bk}) \rightarrow \bk.$
It is well known that these algebra morphisms correspond to closed points of the affine variety $A^{\vee}_{\bk}$. Finally, to give an irreducible finite dimensional representation of $\Cocar(A)$ (equivalently an irreducible finite dimensional $\bk[\Cocar(A)]$-module) amounts to giving an element of the dual torus $A^{\vee}_{\bk}$: more precisely, we have a bijection  
$$A^\vee_{\bk} \longleftrightarrow \Irr\left(\bk[\Cocar(A)]\modu\right)$$
where $\Irr\left(\bk[\Cocar(A)]\modu\right)$ denotes the set of isomorphism classes of irreducible finite dimensional $\bk[\Cocar(A)]$-modules.

\begin{rk} \label{remarque tore dual d'ordre fini}
Considering a trivialization $A^{\vee}_{\bk} \cong (\bk^\ast)^r$, an element of $A^{\vee}_{\bk}$ can be viewed as an $r$-tuple $(t_1,\ldots, t_r)$ of elements of $\bk^\ast$. Since $\bk$ is the inductive limit of its finite subfields, each of these $t_i$ can be viewed as an element in a finite field, so each of these is of finite order in $\bk^\ast$. We deduce easily that the element $t\in A^{\vee}_{\bk}$ is of finite order in the torus. Note also that, $\bk$ being algebraically closed, it is a perfect field of positive characteristic. Thus the morphism $x \mapsto x^\ell$ is a field automorphism and so the orders of $x$ and $x^\ell$ are the same. Thus $\ell$ does not divide the order of any $t \in A^\vee_{\bk}$. 
\end{rk}

We let $D^b_{\str}(X, \bk)_{[\underline{t}]}$ be the full subcategory of $D^b_{\str}(X, \bk)$ whose object are those $\fff$ such that the monodromy morphism $\varphi_{\fff} : \bk[\Cocar(A)]\rightarrow \Aut(\fff)$ factors through the quotient $\bk[\Cocar(A)]/\langle e^\lambda - \lambda(t) \mid \lambda \in \Cocar(A)\rangle$. Note that this is not a triangulated subcategory.

%
\section{The Lusztig--Yun monodromic category} \label{section Lusztig Yun}
%

 \subsection{Preliminaries} \label{subsection preliminaries LY}

 Consider an algebraic group $H$ and a $H$-variety $X$. We denote the $H$-equivariant constructible bounded derived category of sheaves of $\bk$-vector spaces on $X$ by $D^b_{H}(X, \bk)$. Assume that there exists a finite subgroup $K$ of $H$, contained in the center of $H$, which acts trivially on $X$. 
 
 Recall that the center $Z(\mathcal{C})$ of an additive category $\mathcal{C}$ is the endomorphism ring of the identity functor $\id_{\mathcal{C}}$. If $\mathcal{C}$ is $\bk$-linear, then $Z(\mathcal{C})$ is a commutative $\bk$-algebra.
 \begin{lem} \label{lemme A agit sur Id}
 The finite abelian group $K$ acts functorially on the identity functor of $D^b_{H}(X, \bk)$. Said otherwise, we have a $\bk$-algebra morphism 
 $$\bk[K]\rightarrow Z(D^b_{H}(X, \bk)).$$

 \end{lem}
 The lemma means that for any $\fff \in D^b_{H}(X, \bk)$ and any $a \in K$, we have an automorphism $\vartheta_{ \fff, a}$ of $\fff$, and moreover these automorphisms are compatible with the group structure of $K$.
 
 \begin{proof}
We recall one of the incarnations of the equivariant derived category as fibered category (see \cite[\S 2.4.3]{BL}). One can view an object $\fff \in D^b_{H}(X, \bk)$ as a collection of objects $\fff(P) \in D^b(P/H)$ for every resolution $P \rightarrow X$. Moreover, for any morphism $P \xrightarrow{\rho} Q\rightarrow X$ of resolutions, we have isomorphisms $\overline{\rho}^\ast \fff(Q) \cong \fff(P)$, satisfying some usual compatibility conditions (where $\overline{\rho}$ is the induced morphism on quotients).
To obtain such a collection from an object $\fff$ in $D^b_{H}(X, \bk)$, one proceeds in the following way: for any resolution $p : P \rightarrow X$, we have an object $p^\ast\fff$ in $D^b_{H}(P, \bk)$ since $p$ is $H$-equivariant; this in turn defines an object $\fff(P)\in D^b(P/H, \bk)$ since $P$ is a free $H$-space. The definition of the isomorphisms $\overline{\rho}^\ast \fff(Q) \cong \fff(P)$ is obvious. 

Now, consider two resolutions $p : P \rightarrow X$ and $q : Q \rightarrow X$ and a morphism of resolutions $\rho : P \rightarrow Q$. We assume that we have a group $L$ acting on $P$ and $Q$, commuting with the $H$-action, making $\rho$ a $L$-equivariant map and $p,q$ two $L$-invariant maps. Then the objects $q^\ast \fff$ and $p^\ast \fff$ defines objects in $D^b_{H \times L}(P, \bk)$.
Similarly, the isomorphism $\rho^\ast q^\ast \fff \cong p^\ast \fff$ lies in the category $D^b_{H \times L}(P, \bk)$. In order to get $\overline{\rho}^\ast \fff(Q) \cong \fff(P)$, we identify $D^b_{H}(P, \bk)$ with $D^b(P/H, \bk)$. Using the equivalence $D^b_{H\times L}(P, \bk) \cong D^b_{H\times L/ H \times \{1\}}(P/(H \times \{1\}), \bk)= D^b_{L}(P/H, \bk)$,
we thus see that this is actually an isomorphism in $D^b_{L}(P/H, \bk)$.

Consider $T = H \times X \xrightarrow{a} X$ and $T' = H \times H \times X \xrightarrow{a'}X$ with $a' (h,k,x) = hk\cdot x$.
We have the following action of $H\times H$ on $T'$ and $T$, respectively given by: 
$$(h_1, h_2) \cdot (a, b, x) = (h_1a h_2^{-1}, h_2 bh_2^{-1}, h_2 \cdot x), \qquad (h_1, h_2) \cdot (a,  x)  = (h_1 ah_2^{-1}, h_2\cdot x).$$
The maps $a$ and $a'$ define two $H \cong H \times \{1\}$-resolutions of $X$. 
We also have two $H$-resolution morphisms $T' \xrightarrow{\varphi_1, \varphi_2}T$ given by $\varphi_1(h,k,x)= (hk, x)$ and $\varphi_2(h,k,x)= (h, k\cdot x)$. The induced maps on quotients are given respectively by $\pr_2$ and $a : H \times X\rightarrow X$. Thus we obtain isomorphisms
$$\pr_2^\ast\fff(T) \cong \fff(T') \cong a^\ast\fff(T).$$
By definition, this can be written as 
\begin{equation} \label{equat2}
     \vartheta_{\fff}  : \pr_2^\ast(\For(\fff)) \xrightarrow{\sim} a^\ast(\For(\fff)).
\end{equation}
It is clear by definition that this construction is functorial, and that the isomorphism $\vartheta_{\fff}$ satisfies the usual cocycle condition.

The maps $\varphi_i$ for $i=1,2$ are $H \times H$-equivariant for these actions, moreover, the maps $a$ and $a'$ are $\{1\}\times H$-invariant. We deduce from the discussion above that the isomorphisms $\overline{\varphi_i}^\ast\For(\fff) \xrightarrow{\sim} \For(\fff)$ define isomorphisms in $D^b_{H}(H \times X, \bk)$, where the latter category is defined with respect to the action of $H$ on $H \times X$ given by $h \cdot(a, x) = (hah^{-1}, h\cdot x)$.

If $k$ belongs to $K \subseteq Z(H)$, then the map $X \xrightarrow{\sim} \{k\}\times X \hookrightarrow H\times X$ is $H$-equivariant, thus the pullback $\vartheta_{\fff, k}$ of $\vartheta$ along this map gives an automorphism of $\fff$ in $D^b_{H}(X, \bk)$. The cocycle condition for $\vartheta_{\fff}$, ensures that the $\vartheta_{\fff, k}$'s respect the group structure of $K$, i.e.~that $\vartheta_{\fff, k} \circ \vartheta_{\fff, k'} = \vartheta_{\fff, kk'}$ for $k, k' \in K$.
This concludes the proof of the lemma.
 \end{proof}


\subsection{Equivariant categories} \label{subsection equivariant monodromic LY}


Consider a finite central isogeny $\Tilde{H}\xrightarrow{\nu}H$ with kernel $K$ and $\chi : K \rightarrow \bk^\ast$ any character of $K$. We can obviously look at $X$ as a $\Tilde{H}$-variety. We obtain in this way an equivariant bounded derived category $D^b_{\Tilde{H}}(X,\bk)$. Thanks to lemma \ref{lemme A agit sur Id}, the finite abelian group $K$ acts on the identity functor of this category. We can thus consider the full subcategory consisting of objects $\fff$ such that $\vartheta_{\fff, a} = \chi(a)\id_{\fff}$ for all $a \in K$ (this is indeed a subcategory thanks to the functoriality property of the $\vartheta$'s, cf. lemma \ref{lemme A agit sur Id}). We denote this subcategory $D^b_{\Tilde{H}, \chi}(X, \bk)$. 

We consider now a (fixed) multiplicative rank-one $\bk$-local system $\fll$ on $H$. This means that the pullback of $\fll$ along the multiplicative map $H\times H \rightarrow H$ is $\fll \boxtimes \fll$. Assume that there exists a finite central isogeny $\nu$ with kernel $K$ of cardinality prime to $\ell$ and a character $\chi$ of $K$ such that 
$$\fll = (\nu_\ast \fbk_{\Tilde{H}})_{\chi}.$$
Let us explain a bit what this equality means. The isogeny $\nu$ is a $K$-equivariant map (for $H$ endowed with the trivial $K$-action). We can then define an action of $K$ on the pushforward $\nu_{\ast}\fbk_{\widetilde{H}}$; then $(\nu_\ast \fbk_{\Tilde{H}})_{\chi_{\small{\fll}}}$ denotes the $\chi_{\small{\fll}}$-isotypic component.
 We will consider the category $D^b_{\Tilde{H}, \chi}(X, \bk)$. For now, the reader has probably noticed that the definition of $D^b_{\Tilde{H}, \chi}(X, \bk)$ requires some choices, for $\nu$ and $\chi$. 
 We will prove that under suitable assumption on the characteristic of $\bk$, this category in fact does not depend on the choice made.
 
To finish this section, let us remark the following facts. Lemma \ref{lemme A agit sur Id} states that the action of $K$ on the equivariant derived category gives an algebra morphism $\bk[K]\rightarrow Z(D^b_{\widetilde{H}}(X, \bk))$. 
Since the cardinality of $K$ is prime to $\ell$, we have a decomposition $\bk[K] = \bigoplus_{\chi}\bk$ indexed by the characters of $K$. Moreover, for each summand $\bk$, we have an idempotent $e_\chi$ and a decomposition $\id_{\bk[K]} = \sum_{\chi}e_{\chi}$.
The identity $I$ in the second term (i.e. the identity of $\id_{D^b_{H}(X, \bk)}$) is the image of $\id_{\bk[K]}$ so we obtain a decomposition $I = \sum_{\chi} E_{\chi}$ with each $E_{\chi}$ idempotent. Since $D^b_{H}(X, \bk)$ is Karoubian (see e.g. \cite[\S 1, Theorem]{LC}), any object $\fff$ decomposes as a direct sum $\fff = \bigoplus_{\chi} \fff_{\chi}$ and $K$ acts via the character $\chi$ on $\fff_{\chi}$. This direct sum splitting is obviously functorial. Thus $D^b_{\widetilde{H}, \chi}(X, \bk)_{\small{\fll}}$ is a direct summand subcategory in $D^b_{\widetilde{H}}(X, \bk)$.
\begin{rk}
The above comment tells us that if there exists a nonzero morphism between two indecomposable objects of $D^b_{\widetilde{H}}(X, \bk)$, then these two objects are in the same direct summand subcategory; in particular each summand inherits the structure of triangulated category from $D^b_{\widetilde{H}}(X, \bk)$.
This also implies the following important result: if we have a $t$-structure on $D^b_{\widetilde{H}}(X, \bk)$ and if $\fff$ belongs to the direct summand subcategory associated to a character $\chi$ of $K$, then the cohomology objects of $\fff$ belong to the same summand.
\end{rk}

\subsection{Well-definiteness: first step} \label{section general consideration}

We start by making some general comments. Consider a connected algebraic group $H$ acting on a variety $X$ and a
 multiplicative rank-one local system $\fll$ on $H$. Assume that we have two finite central isogenies $\widetilde{H_i} \xrightarrow{\nu_i} H$ with kernels $K_i$ satisfying $\gcd(| K_i|, \ell) = 1$ for $i = 1,2$ and characters $\chi_1, \chi_2$ of $K_1$ and $K_2$ respectively such that $((\nu_i)_\ast(\fbk_{\widetilde{H_i}}))_{\chi_i} = \fll$.
 We can then define the $\chi_1$ and $\chi_2$ equivariant categories 
 $$D_1 : = D^b_{\widetilde{H_1}, \chi_1}(X, \bk) \subseteq D^b_{\widetilde{H_1}}(X, \bk) \quad \mbox{and}\quad D_2 : = D^b_{\widetilde{H_2}, \chi_2}(X, \bk)\subseteq D^b_{\widetilde{H_2}}(X, \bk).$$
 The question is: are these categories (canonically) equivalent ? We begin by reducing to a somehow simpler case: the canonical map $\nu_1\circ \pr_1 : \widetilde{H_1}\times_H \widetilde{H_2} \rightarrow H$ is again a finite central isogeny (here the fibered product has the group structure inherited from the one of the direct product $\widetilde{H_1}\times \widetilde{H_2}$), and its kernel is $\ker(\nu_1) \times \ker(\nu_2)$ so its cardinality is prime to $\ell$. If we can show that the ``character-equivariant" category defined with respect to $(\nu_1\circ \pr_1, \chi_1 \circ {\pr_1}_{\mid \pr_1^{-1}(\ker(\nu_1))})$ is canonically equivalent to $D_1$ then we will get $D_1 \cong D_2$ (canonically). Indeed, by definition we have $\nu_1\circ \pr_1 = \nu_2 \circ \pr_2$; one can check that $\chi_1 \circ {\pr_1}_{\mid \pr_1^{-1}(\ker(\nu_1))}$ and $\chi_2 \circ {\pr_2}_{\mid \pr_2^{-1}(\ker(\nu_2))}$ coincide because they define the same isotypic component $\fll$ in $(\nu_1\circ \pr_1)_\ast\fbk$.
 Therefore we can reduce to the following situation: 
 $$\xymatrix{
 \widetilde{H_1} \ar[r]_{\nu} 
 \ar@/^1pc/[rr]^{\nu_1} 
 & \widetilde{H_2} \ar[r]_{\nu_2} 
 & H 
 }
 $$
 with $\nu$ and $\nu_2$ two finite central isogenies and $\chi_1 = \chi_2 \circ \nu_{\mid \ker(\nu_1)}$. (Note that $\nu_1$ is then itself a finite central isogeny, and the kernels of $\nu_1$, $\nu_2$ and $\nu$ have cardinality prime to $\ell$.)
 The identity of $X$ is a $\nu$-map in the sense of \cite[\S 0.1]{BL}. We can thus consider the equivariant pullback and pushforward functors between the $\widetilde{H_1}$ and $\widetilde{H_2}$ equivariant derived categories on $X$ (these functors are denoted ${Q^\ast_{\id}}$ and ${Q_{\id}}_\ast$ in \cite[Part 6]{BL}). Set $(p_\nu)^\ast:={Q^\ast_{\id}}$. What we finally want to show is that $(p_\nu)^\ast$ induces an equivalence of categories $D_2 \rightarrow D_1$. For the proof, we will need an assumption on the characteristic $\ell$ of $\bk$. The proof goes in two steps: first full faithfulness, then essential surjectivity; but let us note that we will need an assumption on $\ell$.

\subsection{Equivariant cohomology for algebraic groups}

In the following lemma, we consider the notion of torsion primes for a reductive group; we refer to \cite[I.4.3, I.4.4]{SS} for the definition. Let us just recall that for a reductive group $L$, the torsion primes are the torsion primes of the simply-connected cover of its derived subgroup $\mathscr{D}(L)$, together with the primes dividing the order of the fundamental group of $\mathscr{D}(L)$. 

 Any connected algebraic group $H$ over $\C$ can be written as a semidirect product $H = R_u(H) \rtimes L$ with $R_u(H)$ the unipotent radical of $H$ and $L$ a connected reductive subgroup (see e.g. \cite[\S VIII.1, Theorem 4.3]{Ho}). We call $L$ a Levi factor of $H$; we have an isomorphism $L \cong H/ R_u(H)$.

 \begin{lem} \label{cohomology equiv = cohomology reduc}
  Consider $H$ a connected complex algebraic group, and fix a Levi decomposition $ H = R_u(H)\rtimes L$. We fix a maximal torus $T$ of $L$ and we let $W$ denote the Weyl group of $(L, T)$. Finally assume that the characteristic $\ell$ of $\bk$ is not a torsion prime for $L$.
  Then the natural morphism $\mathbf{H}^{\bullet}_{H}(\pt, \bk) \rightarrow \mathbf{H}^{\bullet}_{L}(\pt, \bk)$ is an isomorphism, and the natural morphism 
  $$\mathbf{H}^{\bullet}_{L}(\pt, \bk) \rightarrow\mathbf{H}^{\bullet}_{T}(\pt, \bk)$$ induces an isomorphism $\mathbf{H}^{\bullet}_{L}(\pt, \bk) \rightarrow\mathbf{H}^{\bullet}_{T}(\pt, \bk)^W$.
 \end{lem}

 \begin{proof}[Proof sketch]
  The $H$-equivariant cohomology of the point can be computed as the total cohomology of the space $BH= EH/H$ where $EH$ is any contractible free $H$-space. Note that we have a split surjection 
$$\xymatrix{
H \ar@{->>}[r] &  L. \ar@{-->}@/^-1pc/[l]
}
$$
Thus the space $EH$ can be viewed as a (contractible and free) $L$-space and then chosen as $EL$, i.e. we have $BL= EH/L$. We then obtain a locally trivial map 
$$q  : BL \twoheadrightarrow BH$$ 
whose fibers are isomorphic to the unipotent radical of $H$, so in particular, isomorphic to an affine space. Thus we have $\mathbf{H}^\bullet_{H}(\pt, \bk) \cong \mathbf{H}^\bullet_{L}(\pt, \bk)$.  The last statement of the lemma is proved in \cite[Theorem 1.3]{To}; this is where the assumption on $\car(\bk)$ is necessary.
 \end{proof}

\subsection{Isogenies and equivariant cohomology}
We start by an easy general lemma:
 \begin{lem} \label{lemme ordre conoyau} Consider $\widetilde{H}, H$ two algebraic tori.
 Assume that $\widetilde{H} \xrightarrow{\nu} H$ is an isogeny with kernel $K$ of order $k$, and let $k = p^{a_1}_1\cdots p^{a_r}_r$ be its decomposition as a product of prime numbers.
Then the cokernel $A$ of the induced map 
 $$\nu^{\#} : \Car(H)\rightarrow \Car(\widetilde{H})$$ is finite and its order is of the form $p^{b_1}_1\cdots p^{b_r}_r$ for some non-negative integers $b_i$.
 \end{lem}
 
Consider now two complex algebraic groups $\widetilde{H}$ and $H$ and a finite central isogeny $\nu : \widetilde{H}\rightarrow H$; denote its kernel by $K$.
We can find a Levi decomposition $\widetilde{H}= R_u(\widetilde{H}) \rtimes \widetilde{L}$ (see the discussion before lemma \ref{cohomology equiv = cohomology reduc}).
\begin{lem} \label{lemme fini centre dans reductif}
 Any finite subgroup of the center $Z(\widetilde{H})$ is contained in $\widetilde{L}$. In particular, $K \subseteq \widetilde{L}$. 
\end{lem}

\begin{proof}
 In fact, we will show that any element of finite order in $Z(\widetilde{H})$ is in $\widetilde{L}$. Consider such an $x$, and let $n$ be its order. We can write $x = rs$ with $r \in R_u(\widetilde{H})$ and $s \in \widetilde{L}$. Since $x$ is central, we have in particular $x=s x s^{-1} = sr$, so $r$ and $s$ commute. We then have $1 = x^n = r^n s^n$. These equalities imply that $r^n$ is in $R_u(\widetilde{H}) \cap \widetilde{L}$, thus is trivial. So $r$ is of finite order, and then semisimple. We then obtain $r=1$ (as $r$ is both unipotent and semisimple) and $x = s \in \widetilde{L}.$
\end{proof}
We deduce from lemma \ref{lemme fini centre dans reductif}
 that the map $\nu$ identifies with the natural projection
 $$ R_u(\widetilde{H})\rtimes \widetilde{L} \longrightarrow R_u(\widetilde{H})\rtimes \widetilde{L}/K.$$
 We see in particular that two connected isogeneous groups over $\C$ have isomorphic unipotent radicals. 
We now would like to link the torsion primes for Levi factors of $H$ and $\widetilde{H}$; in fact we show that the torsion primes for $\widetilde{L}$ are already torsion primes for $L$. We start by choosing a maximal torus $\widetilde{T}$ in $\widetilde{L}$; we obtain this way a maximal torus $T = \widetilde{T}/K$ in $L$. As we saw above, the isogeny $\nu$ induces an isogeny at the level of Levi factors. Moreover, the restriction $\nu_{\mid \mathscr{D}(\widetilde{L})}$ of $\nu$ to the derived (semi-simple) subgroup of $\widetilde{L}$ lands in $\mathscr{D}(L)$. We have an induced map on maximal tori $\widetilde{T}\rightarrow \widetilde{T}/K=T$; in turn this gives an injective group morphism $\Car(T) \hookrightarrow \Car(\widetilde{T})$. We finally obtain a surjection of groups $\pi_1(\mathscr{D}(L))=A/\Car(T)\twoheadrightarrow A/\Car(\widetilde{T})= \pi_1(\mathscr{D}(\widetilde{L}))$, where $A$ is the abstract weight lattice of the root system of $\mathscr{D}(L)$ and $\mathscr{D}(\widetilde{L})$. Thus the prime numbers dividing the order of the fundamental group of $\pi_1(\mathscr{D}(\widetilde{L}))$ divides the order of $\pi_1(\mathscr{D}(L))$.

\begin{lem} \label{lemme cohomology H equi}
Take two algebraic groups $\widetilde{H}$ and $H$ and a finite central isogeny $\nu : \widetilde{H}\rightarrow H$; denote its kernel by $K$. Assume \begin{itemize}
    \item the order of $K$ is prime to $\car(\bk) = \ell$, 
    \item $\ell$ is not a torsion prime for the reductive group $H/R_u(H)$.
\end{itemize}

Then the natural morphism 
$$ \mathbf{H}^\bullet_{H}(\pt, \bk)\rightarrow \mathbf{H}^\bullet_{\widetilde{H}}(\pt, \bk) $$
is an isomorphism.
\end{lem}

\begin{proof} 

According to lemmas \ref{cohomology equiv = cohomology reduc} and \ref{lemme fini centre dans reductif}, we can assume that $\widetilde{H}$ and $H$ are reductive. As above we consider a maximal torus $\widetilde{T}$ in $\widetilde{H}$ and the maximal torus $T = \widetilde{T}/K$ under $\widetilde{T}$ in $H$. The isogeny $\nu$ induces an identification of the Weyl groups $\widetilde{W}$ and $W$ of $\widetilde{H}$ and $H$. The induced map $\widetilde{T}\rightarrow T$ is still an isogeny whose kernel has order prime to $\ell$. 
The map $\nu$ induces an injective morphism of abelian group $\Car(T) \rightarrow \Car(\widetilde{T})$ whose cokernel is a finite group of order prime to $\ell$ according to lemma \ref{lemme ordre conoyau}. Taking tensor product with $\bk$, we obtain an isomorphism 
 $\Car(T)\otimes_{\Z} \bk  \xrightarrow{\sim}\Car(\widetilde{T})\otimes_{\Z} \bk$ and then an isomorphism
\begin{equation}\label{iso cohomology T equiv} \Sym\left(\Car(T)\otimes_{\Z} \bk\right)  \xrightarrow{\sim}\Sym(\Car(\widetilde{T})\otimes_{\Z} \bk).
\end{equation}
Using the (well-know) fact that for a torus, we have $\mathbf{H}^{\bullet}_{T}(\pt)\cong \Sym\left(\Car(T)\otimes_{\Z} \bk\right)$, we obtain 
\begin{equation}
    \mathbf{H}^{\bullet}_{T}(\pt)\cong \mathbf{H}^{\bullet}_{\widetilde{T}}(\pt).
\end{equation}

The groups $\widetilde{W}$ and $W$ act on the corresponding characters groups $\Car(\widetilde{T})$ and $\Car(T)$. Moreover, the map $\Car(T) \rightarrow \Car(\widetilde{T})$ induced by the isogeny commutes with the $\widetilde{W}$ and $W$ actions.
The action of $W$ on $\Sym(\Car(T)\otimes_{\Z} \bk)$ induces the $W$-action on $\mathbf{H}^\bullet_{H}(\pt, \bk)$ (and similarly for $\widetilde{H}$).
As discussed above the statement of the lemma, since $\car(\bk)$ is not a torsion prime $H$, it is not torsion for $\widetilde{H}$. Thanks to lemma \ref{cohomology equiv = cohomology reduc} we have 
 $\mathbf{H}^{\bullet}_{H}(\pt, \bk) \cong \Sym(\Car(T)\otimes_{\Z} \bk)^{W}$ and $\mathbf{H}^{\bullet}_{\widetilde{H}}(\pt, \bk) \cong \Sym(\Car(\widetilde{T})\otimes_{\Z} \bk)^{\widetilde{W}}$. Thus the isomorphism \eqref{iso cohomology T equiv} induces an isomorphism 
 $$\mathbf{H}^{\bullet}_{\widetilde{H}}(\pt, \bk)\cong \mathbf{H}^{\bullet}_{H}(\pt, \bk).$$
 This concludes the proof of the lemma.
\end{proof}

\subsection{A spectral sequence for equivariant cohomology} 

 \begin{lem} \label{lemme spectral sequence}
 Consider a $G$-variety $X$, for $G$ a connected algebraic group.
 For an object $\fff \in D^b_G(X, \bk)$ we have a converging spectral sequence
 $$E^{p,q}_2 = \mathbf{H}^p_{G}(\pt, \bk)\otimes_{\bk}\mathbf{H}^q(X, \fff) \Rightarrow \mathbf{H}^{p+q}_G(X, \fff).$$
\end{lem}

 \begin{proof}
 Fix a contractible free $G$-space $EG$.
The projection map $p : EG \times X \rightarrow X$ gives an $\infty$-acyclic resolution of $X$ in the sense of \cite[Definition 1.9.1]{BL} for the action $g\cdot (e, x) = (eg^{-1}, gx)$. The quotient is $EG\times^G X$, whith $BG$ the classifying space; since $G$ is connected, $BG$ is simply connected. Denote by $q : EG \times X \rightarrow EG \times^G X$ the quotient map.
Thanks to \cite[Lemma 2.3.2]{BL} we can consider our object in the equivariant derived category as an object $\fff$ in $D^b(EG\times^G X)$ such that there exists a $\fff_X\in D^b(X)$ and an isomorphism $p^\ast (\fff_X) \cong q^\ast(\fff)$.

 Consider the following diagram 
 $$EG\times^G X \xrightarrow{\pr} BG \xrightarrow{a} \{\pt\},$$
 with $\pr([e,x]) = [e]$.
We have a Grothendieck spectral sequence $R^q a_{\ast} \circ R^p(\pr)_\ast  \Rightarrow R^{p+q}(a\circ \pr)_{\ast}$. We apply it to $\fff$. The object $R^p (\pr)_{\ast}(\fff)$ has constant stalks isomorphic to $\mathbf{H}^p(X, \fff_X)$. We thus obtain a local system on $BG$; the latter space being simply connected, we have in fact a constant local system isomorphic to $ \fbk_{BG} \otimes_{\bk} \mathbf{H}^p(X, \fff) $. When we apply $R^q a_{\ast}$ to this object, we obtain $\mathbf{H}^p_{G}(\pt, \bk)\otimes_{\bk}\mathbf{H}^q(X, \fff)$.
 Since $R^{p+q}(a\circ \pr)_{\ast}(\fff) =\mathbf{H}^{p+q}_G(X, \fff) $, we get the result.
 \end{proof}

\subsection{Full faithfulness of $p^\ast_\nu$} \label{section fully faithfulness pnu} 

We keep the notations of the beginning of \S \ref{section general consideration}.
We show here that under suitable assumptions, the functor $p^\ast_\nu$ is fully faithful. 
\begin{lem} \label{lemme bonne def LY}
With the notations of section \ref{section general consideration}, assume that $\ell$ is not a torsion prime for $H/R_u(H)$.
Then the functor $(p_\nu)^\ast$ is fully faithful.
\end{lem}
\begin{proof}
Consider $\fff, \fgg$ two objects in $D^b_{\widetilde{H_i}}(X, \bk)$; applying lemma \ref{lemme spectral sequence} to $R\sHom(\fff, \fgg)$ one obtains the following converging spectral sequence: 
$$E^{p,q}_2 = \mathbf{H}^p_{\widetilde{H_i}}(\pt, \bk)\otimes_{\bk} \Hom_{D^b(X, \bk)}(\fff, \fgg[q]) \Rightarrow \Hom_{D^b_{\widetilde{H_i}}(X)}(\fff, \fgg[p+q]).$$
On the morphism spaces, the functor $p^\ast_{\nu}$ comes from a morphism of spectral sequences induced by the map 
\begin{equation}\label{application sur les cohomologies equivariantes}
\mathbf{H}^\bullet_{\widetilde{H_2}}(\pt, \bk) \rightarrow \mathbf{H}^\bullet_{\widetilde{H_1}}(\pt, \bk).\end{equation}

According to lemma \ref{lemme cohomology H equi}, the morphism \eqref{application sur les cohomologies equivariantes} is an isomorphism (recall that the cardinality of the kernels of $\nu_1$, $\nu_2$ and $\nu$ are prime to $\ell$). This allows us to conclude that $p^\ast_\nu$ is fully faithful, and thus lemma \ref{lemme bonne def LY} is proved.
\end{proof}

\subsection{Essential surjectivity of $p^\ast_\nu$}\label{subsection essential surjectivity}

We now show the essential surjectivity of $p^\ast_\nu$. We keep the notations of \S \ref{section general consideration}. We also keep the assumptions of lemma \ref{lemme cohomology H equi} on $\ell$; thanks to lemma \ref{lemme bonne def LY}, the functor $p^\ast_{\nu} : D_2 \rightarrow D_1$ is fully faithful.
Since any object in $D_1$ is an extension of its perverse cohomology objects, using lemma \ref{lemme bonne def LY}, in order to show that $p^\ast_\nu$ is essentially surjective, we only need to check that the perverse sheaves in $D_1$ are in the essential image, thanks to the following well known fact:
\begin{center}
    ``Assume that $F : \mathcal{C}\rightarrow \mathcal{D}$ is a functor between triangulated categories which is full, and that the essential image of $F$ contains a family of objects generating $\mathcal{D}$ (as a triangulated category), then $F$ is essentially surjective."
\end{center}
Set $K := \ker(\nu)$. In order for a perverse sheaf $\fff$ on $X$ to define a $\widetilde{H_i}$-equivariant object, we just need the existence of an isomorphism between the pullbacks along the projection map and the action map (see \cite[\S A.1, 3rd ``reasonable definition"]{BaR}). Now if $\fff$ is a perverse sheaf in $D_1$, then by definition it carries a $\widetilde{H}_1$-equivariant structure such that $K$ acts trivially (via the map of lemma \ref{lemme A agit sur Id}). Indeed, by definition, we have that the group $K_1 = \ker(\nu_1) = \nu^{-1}(\ker(\nu_2))$ acts on $\fff$ via the character $\chi_1 = \chi_2 \circ \nu_{\mid \ker(\nu_1)}$. But $K \subseteq \ker(\chi_1)$, so $K$ acts trivially on $\fff$.

One then remarks that the isomorphism $\vartheta$ between the two pullbacks of $\fff$ to $\widetilde{H}_1\times X$ along the action map and the projection $\widetilde{H}_1\times X \rightarrow X$
is the identity map of $({\pr_2}_{\mid K \times X})^\ast \fff$ when restricted to $K \times X$, hence it descends to an isomorphism on $\widetilde{H}_2 \times X = \widetilde{H}_1/K \times X$. 
Let us elaborate on this fact: the perverse sheaves form a stack for the \'etale topology (see \cite[\S 2.2.19]{BBD}). Therefore to see that our isomorphism descends to an isomorphism on $\widetilde{H}_2 \times X$, we have to show that its pullbacks under the two projections $\pi_1, \pi_2 : ( \widetilde{H_1}\times X) \times_{\widetilde{H_2}\times X} ( \widetilde{H_1}\times X) \rightarrow  \widetilde{H_1}\times X$ coincide. 

\begin{lem}
We have $\pi_1^\ast \vartheta = \pi_2^\ast \vartheta$.
\end{lem}
\begin{proof}
The different maps that we will consider are depicted in the following (non commutative) diagram.
\vspace{-0.2cm}
 $$ {\small \xymatrix{
\widetilde{H_1}\times \widetilde{H_1}\times X  \ar@<10pt>[rrr]^{\id_{\widetilde{H_1}} \times a} \ar@<-10pt>[rrr]^{m\times \id_X} \ar[rrr]^{\pr_{2,3}} && & \widetilde{H_1}\times X \ar@<2pt>[rr]^a \ar@<-2pt>[rr]_{\pr_2} &&  X \\
\widetilde{H_1}\times K \times X\ar@{^{(}->}[u]_{\widetilde{\iota}} \ar[rrr]^{\pr_{2,3}} &&& K \times X \ar@{^{(}->}[u]_{\iota} \ar@<2pt>[urr] \ar@<-2pt>[urr] &&
 }}
 $$
(the two maps from $K \times X$ to $X$ are given by $a \circ \iota$ and $\pr_2 \circ \iota$).
The isomorphism $\vartheta$ satisfies the cocycle condition 
$ (m \times \id_X)^\ast \vartheta = (\id_{\widetilde{H_1}}\times a)^\ast \vartheta \circ \pr_{2,3}^\ast \vartheta.$
 We want to show that the following diagram commutes: 
 $$\xymatrix{
 \pi_1^\ast a^\ast (\fff) \ar@{=}[r] \ar[d]^{\pi_1^\ast \vartheta} &   \pi_2^\ast a^\ast (\fff)  \ar[d]^{\pi_2^\ast \vartheta} \\
 \pi_1^\ast \pr_2^\ast \fff \ar@{=}[r] &\pi_2^\ast \pr_2^\ast \fff.
 }
 $$

  We consider the diagram
 \begin{equation}\label{diagram fibré}{\small \vcenter{\xymatrix{
 \left( \widetilde{H_1}\times X\right) \times_{\widetilde{H_2}\times X} \left( \widetilde{H_1}\times X\right) \ar@<2pt>[dr]^{\pi_1} \ar@<-2pt>[dr]_{\pi_2}  && \widetilde{H_1}\times K \times X \ar[ll]_-{\sim} \ar@<2pt>[dl]^{m_K \times \id_X} \ar@<-2pt>[dl]_{\pr_{1,3}} \\
 & \widetilde{H_1}\times X&
 }}}
 \end{equation}
 where the horizontal map is given by $(p, k, x) \mapsto ((p, x), (pk, x))$.

 Thanks to diagram \eqref{diagram fibré}, we are reduced to the commutativity of
\begin{equation}\label{diagram fibré iso}\vcenter{\xymatrix{
 \pr_{1,3}^\ast a^\ast (\fff) \ar@{=}[r] \ar[d]^{\pr_{1,3}^\ast \vartheta} &   (m_K\times \id_X)^\ast a^\ast (\fff)  \ar[d]^{(m_K\times \id_X)^\ast \vartheta} \\
 \pr_{1,3}^\ast \pr_2^\ast \fff \ar@{=}[r] &(m_K\times \id_X)^\ast \pr_2^\ast \fff.
 }}
 \end{equation}
 But now, we use the isomorphism $(m_K \times \id_X)^\ast \cong \widetilde{\iota}^\ast \circ (m\times \id_X)^\ast$ and the cocycle condition to find that 
 \begin{align*}
     (m_K \times \id_X)^\ast \vartheta  & = \widetilde{\iota}^\ast (\id_{\widetilde{H_1}}\times a)^\ast \vartheta \circ \widetilde{\iota}^\ast \pr_{2,3}^\ast \vartheta \\
     & = \pr_{1,3}^\ast\vartheta \circ \underbrace{\pr_{2,3}^\ast \iota^\ast \vartheta}_{\id}
 \end{align*}
 by the assumption that $\vartheta_{\mid K \times X}= \id$.
 Thus the diagram \eqref{diagram fibré iso} commutes and we are done.
\end{proof}

We see that $\fff$ is a $\widetilde{H}_2$-equivariant perverse sheaf, thus it defines an object in the $\widetilde{H}_2$-equivariant bounded category. Moreover, the isomorphism 
\begin{equation}\label{cette equation}
a_2^\ast\fff \rightarrow \pr_2^\ast\fff
\end{equation}(where $a_2$ is the action map for $\widetilde{H}_2$ acting on $X$) on $\widetilde{H}_2\times X$ comes from the one on $\widetilde{H}_1 \times X$ (i.e. is the descent of $\vartheta$ to $\widetilde{H}_2\times X$). Since $\nu : \widetilde{H}_1 \rightarrow \widetilde{H}_2$ is an isogeny and hence is surjective, it is easy to see that the action of $K_2$ on $\fff$ via \eqref{cette equation} comes from the action of $K_1$ on $\fff$. So in particular, one deduces that $K_2$ acts via $\chi_2$ on $\fff$. Thus our functor is essentially surjective.

\begin{defn}
In the above situation, we set 
$$D^{\mathrm{LY}}_{ H}(X, \bk)_{\small{\fll}} : =  D^b_{\widetilde{H}, \chi_{\small{\fll}}}(X, \bk)$$ and call this category the Lusztig--Yun monodromic equivariant category of monodromy $\fll$.
\end{defn}

\begin{rk}
In the case where $X$ admits a stratification $\str$, we can 
make all the preceding construction replacing the category $D^b_H(X, \bk)$ by the $\str$-constructible derived equivariant category $D^b_{H, \str}(X, \bk)$. 
We obtain the $\str$-constructible Lusztig--Yun equivariant monodromic category $D^{\mathrm{LY}}_{ H, \str}(X, \bk)_{\small{\fll}}$.
\end{rk}

Let us fix an isogeny $\widetilde{H}\xrightarrow{\nu} H$ and a character $\chi$ of its kernel such that we can view $D^{\mathrm{LY}}_{H}(X, \bk)$ as a full subcategory of $D^b_{\widetilde{H}}(X, \bk)$. One can consider the restriction of the forgetful functor $\For : D^b_{\widetilde{H}}(X, \bk) \rightarrow D^b_{\str}(X, \bk)$ to $D^{\mathrm{LY}}_{H}(X, \bk)_{\small{\fll}}$. Moreover, this functor does not depends of the choices made for $\nu$ and $\chi$, in the following sense. With the notations of \S \ref{section general consideration}, we have an isomorphism $\For_1 \circ (p_\nu)^\ast \cong \For_2$ of functors $D_2 \rightarrow D^b(X, \bk)$ (where $\For_i$ is the forgetful functor $D_i \rightarrow D^b(X, \bk)$ for $i=1,2$).


\section{Perverse monodromic sheaves on stratified $T$-varieties} \label{section Perverse monodromic sheaves on stratified $T$-varieties}



\subsection{Local systems on $T$}

We come back to the main interest of the present paper: consider a torus $T$ and $(X, \str)$ any stratified $T$-variety with $T$-stable strata. 

According to \S \ref{subsection : Setting and Notations}, the isomorphism classes of irreducible local systems on $T$ (which correspond to irreducible representations of $\Cocar(T)$) are in bijection with elements of $T^\vee_{\bk}$. 
For $t \in T^\vee_{\bk}$, we denote by $\fll^T_{t}$ the rank-one local system on $T$ corresponding to the $\Cocar(T)$-representation $L_t := \bk[\Cocar(T)]/\langle e^\lambda - \lambda(t) \mid \lambda \in \Cocar(T)\rangle$.
Reasoning in terms of associated representations, it is not difficult to see that $\fll^T_t$ is multiplicative.

\begin{lem} \label{isotypic component of the pushforward}
For any simple local system $\fll^T_t$ on $T$ there exists a finite central isogeny $\nu : \widetilde{T}\rightarrow T$ and a character $\chi$ of $K:= \ker(\nu)$ such that
\begin{enumerate}
    \item the cardinality of $K$ is prime to $\ell$, 

\item $\fll^T_t = \left(\nu_\ast(\fbk_{T})\right)_\chi,$
the isotypic component of $\nu_\ast(\fbk_{T})$ on which $K$ acts via $\chi$.
\end{enumerate}
\end{lem}
\begin{proof}
 Denote by $n$ the order of $t$ (recall that $t$ is of finite order thanks to remark \ref{remarque tore dual d'ordre fini}, and that this order $n$ is prime to the characteristic $\ell$ of $\bk$). We consider the map $e_n : T \rightarrow T, \qquad x\mapsto x^n$; this is a finite central isogeny with kernel $K_n:= \ker(e_n)$ of order $n^{\rak(T)}$, which is prime to $\ell$. The pushforward of the constant local system along $e_n$ is still a local system, thus we can make all our calculations in the representation-theoretic world; $(e_n)_\ast(\fbk_{T})$ identifies with $\bk[\Cocar(T)]\otimes_{\bk[\Cocar(T)]}\bk$
where the structure of $\bk[\Cocar(T)]$-module on itself in the tensor product is given by $e^\lambda \cdot P = e^{n\lambda}P$.
One easily checks that this tensor product identifies with 
$$\bk[\Cocar(T)/e_n^\#(\Cocar(T))]$$ ($e_n^\# : \Cocar(T) \rightarrow \Cocar(T)$ is the map induced by $e_n$). Since $\Cocar(T)/e_n^\#(\Cocar(T))$ is a finite abelian group of order $n^{\rak(T)}$ (so in particular prime to the characteristic $\ell$ of $\bk$), its category of $\bk$-representations is semisimple and we can decompose $\bk[\Cocar(T)/e_n^\#(\Cocar(T))]$ as a direct sum of simple dimension one representations, each one associated to a character of this quotient. Thus the associated local system decomposes as a direct sum of rank one local systems according to the characters of $\Cocar(T)/e_n^\#(\Cocar(T))$. 

We have a non-zero adjunction morphism   
$\fll^T_{t} \rightarrow (e_n)_\ast (e_n)^\ast \fll^T_{t}.$
Using the diagram in the statement of theorem \ref{theoreme representation systèmes locaux}, it is easy to see that $e_n^\ast(\fll^T_{t}) = \fll^T_{t^n} \cong \fbk_T$. With the preceding discussion, this adjunction morphism has to be the inclusion of a direct summand and this direct summand is  
obviously associated to the character of evaluation at $t$
$$ ev(t) : \Cocar(T)/e_n^\#(\Cocar(T)) \rightarrow \bk^\ast, \quad [\lambda]\mapsto \lambda(t).$$
This character is well defined precisely because $t$ is of order $n$.

One easily checks that evaluation at $e^{\frac{2i\pi}{n}}$ gives a (well-defined) isomorphism 
$\varphi :\Cocar(T)/e_n^\#(\Cocar(T))\xrightarrow{\sim} K_n. $
The morphism $ev(t) \circ \varphi^{-1}$ gives the wished-for character $\chi$ of $K_n$.

\end{proof}

According to lemmas \ref{lemme bonne def LY} and \ref{isotypic component of the pushforward}, for any $T$-variety $X$, we can consider without any ambiguity the category $D^{\mathrm{LY}}_{ T}(X, \bk)_{\fll^T_{t}}$ for any $t \in T^\vee_{\bk}$.

\subsection{Monodromy of local systems on $T$} 

Consider a local system $\fll$ on $T$ and denote by $L$ its associated $\Cocar(T)$-representation. The canonical monodromy action on $\fll$ gives under the equivalence of theorem \ref{theoreme representation systèmes locaux} an action of $\Cocar(T)$ on $L$ defined by a morphism 
$$\varphi_{\small{\fll}} : \Cocar(T)\rightarrow  \End_{\Cocar(T)}(L).$$

\begin{lem} \label{monodromy des sytème locaux}
With the above notations, we have 
$\varphi_{\small{\fll}}(\lambda) = (e^{\lambda}) \cdot,$
the right-hand side being the action of $\lambda$ on $L$ given by the structure of $\Cocar(T)$-module.
\end{lem}

\begin{proof}
In this proof we denote by $a$ the multiplication map $T \times T \rightarrow T$. 

The monodromy is defined by (the appropriate restriction of) an isomorphism $\iota : \pr_2^\ast\fll \rightarrow a(n)^\ast\fll $ satisfying $\iota _{\mid\{1\}\times T}= \id_{\small{\fll}}$. In view of theorem \ref{theoreme representation systèmes locaux}, this amounts to $\iota_{(1,1)}= \id_{L}$. What we want to determine is $\iota_{\mid(\lambda_n,1)}$ for any $\lambda \in \Cocar(T)$. In the following, we will denote by $\pr_2^\ast L$ and $a(n)^\ast L$ the $\Cocar(T)$-representations associated to $\pr_2^\ast\fll$ and $a(n)^\ast\fll$ respectively. The action of a pair $(\lambda, \mu)$ on $\pr_2^\ast L$ (resp. $a(n)^\ast L$) is given by the action of $\mu$ (resp. $n\lambda + \mu$) on $L$.
 We start by some general considerations. Consider a topological space $X$ satisfying the condition of theorem \ref{theoreme representation systèmes locaux} and a loop $\gamma : [0,1]\rightarrow X$.
The pullback of any local system $\fff$ on $X$ under $\gamma$ gives a local system on $[0,1]$, which is trivial since $[0,1]$ is simply connected. We know that we then have a canonical identification $\Gamma([0,1], \gamma^\ast (\fff)) \cong (\gamma^\ast (\fff))_x$ for any $x\in [0,1]$. Moreover, if $\fhh$ is another local system on $X$ and $f : \fff \rightarrow \fhh$ is an isomorphism, we have a commutative diagram (for any $x,y \in [0,1]$)
{\small\begin{equation}\label{digram} \vcenter{\xymatrix@C=0.3cm@R =0.5cm{
& (\gamma^{\ast}(\fff))_x
\ar[rr]^{(\gamma^\ast(f))_x}_{\sim} && (\gamma^{\ast}(\fhh))_x\\
\Gamma([0,1], \gamma^\ast(\fff))\ar[ru]^{\sim} \ar[rd]_{\sim}&&&& \Gamma([0,1], \gamma^\ast (\fhh))\ar[lu]^{\sim} \ar[ld]_{\sim}\\
&(\gamma^{\ast}(\fff))_y
\ar[rr]^{\sim}_{(\gamma^\ast(f))_y} && (\gamma^{\ast}(\fhh))_y.}}
\end{equation}}
Now, for any $n\geq 1$, let $\gamma_n$ be the path 
$[0,1]\rightarrow \C^\ast, \quad t \mapsto e^{\frac{2i\pi t}{n}}$
and $\gamma^{\lambda}_n$ the loop in $T$ given by $\lambda\circ \gamma_n$.
We consider the diagram \eqref{digram} with $X = T\times T$, $\fff=\pr_2^\ast(\fll)$, $\fhh = a(n)^\ast(\fll) $, the isomorphism $f=\iota$ and $\gamma= \gamma^{\lambda}_n\times \gamma^{0}_1$. Remark that $\gamma^{0}_1$ is just a fancy notation for the constant path at $1 \in T$. Also we choose $x= 0$ and $y=1$. We get
$$\xymatrix{
(\gamma^\ast \pr_2^\ast(\fll) )_0 \ar[rr]^{\id} \ar[d]_{\wr} && (\gamma^\ast a(n)^\ast (\fll))_0 \ar[d]^\wr \\ 
(\gamma^\ast \pr_2^\ast(\fll))_1 \ar[rr]_{\sim}  && (\gamma^\ast a(n)^\ast (\fll))_1.
}
$$
We then remark that $\pr_2 = \pr_2\circ (e_n \times \id_T)$ and that we have $(e_n\times \id_T)\circ \gamma = (\gamma^{\lambda}_1 \times \gamma^0_{1})$.
The left (resp. right) vertical arrow is given (at the level of representations) by the action of $(\gamma^{\lambda}_1 \times \gamma^0_{1})$ on $\pr_2^\ast L$ (resp. $a(n)^\ast L$), which is by definition the action of $ 0$ (resp. $\lambda$) on $L$.  
We can rewrite this diagram in the following form
$$\xymatrix{
(\pr_2^\ast (\fll))_{(1,1)} \ar[rr]^{\id= \iota_{(1,1)}} \ar[d]_{\wr} && (a(n)^\ast (\fll))_{(1,1)} \ar[d]^\wr \\ 
(\pr_2^\ast (\fll))_{(\lambda_n,1)} \ar[rr]^{\varphi_{\small{\fll}}(\lambda)}_{\sim}  && (a(n)^\ast(\fll))_{(\lambda_n,1)}.
}
$$
and we can then conclude the proof. 

\end{proof}

\subsection{Perverse monodromic sheaves} \label{section Perverse monodromic sheaves}
In this section, we give four ``reasonable definitions" of monodromic perverse sheaves, and show that they actually coincide.

Let us fix an isogeny $\nu : \widetilde{T}\rightarrow T$ and a character $\chi$ of $\ker(\nu)$ that allow us to define $D^{\mathrm{LY}}_{T, \str}(X, \bk)_{\small{\fll^T_t}}$ as a full direct summand in $D^b_{\widetilde{T}, \str}(X, \bk)$.
The restriction of the perverse $t$-structure on $D^b_{\Tilde{T}, \str}(X, \bk)$ as defined in \cite[\S 5.1]{BL} to $D^{\mathrm{LY}}_{ T, \str}(X, \bk)_{\fll^T_{t}}$ gives a $t$-structure on this category. (This is true essentially because the category $D^{\mathrm{LY}}_{ \str}(X, \bk)_{\fll^T_{t}}$ is a full direct-summand-subcategory of $D^b_{\Tilde{T}, \str}(X)$.) Note that, by the definition of the perverse $t$-structure on the equivariant bounded category and the remark at the very end of subsection \ref{section fully faithfulness pnu}, the perverse $t$-structure does not depend on the choice of $\nu$ and $\chi$. 
\begin{defn}\label{definition pervers LY} Define the perverse $t$-structure on $D^{\mathrm{LY}}_{\str, T}(X, \bk)_{\fll^T_{t}}$ to be the shift by $r = \dim(T)$ of the usual perverse $t$-structure on $D^b_{\Tilde{T}, \str}(X, \bk)$. This means that $\fff \in D^{\mathrm{LY}}_{ \str}(X, \bk)_{\fll^T_{t}}$ is perverse in this category if and only if $\fff[r]$ is perverse in $D^b_{\Tilde{T}, \str}(X, \bk)$.
The heart is then denoted $\mathrm{P}^{\mathrm{LY}}_{T,\str}(X, \bk)_{\fll^T_{t}}$, and its objects are called the Lusztig--Yun monodromic perverse sheaves. 
\end{defn}

Fix $t \in T^\vee_{\bk}$. Here are the different possibilities for the category of monodromic perverse sheaves with ``monodromy $t$": 
\begin{enumerate}
    \item the heart $\mathrm{P}^{\mathrm{LY}}_{T,\str}(X, \bk)_{\fll^T_{t}}$ of the perverse $t$-structure in the Lusztig--Yun equivariant monodromic category,
    \item the full subcategory $\Perv(X, \bk)_{[\underline{t}]}$ of $\Perv(X, \bk)$ whose objects are those complexes in $D^b_{\str}(X, \bk)_{[\underline{t}]}\cap \Perv(X, \bk)$, 
    \item the category $\Pco(X, \bk)$ whose objects are pairs $(\fff, \vartheta_{\fff})$ with $\fff$ a perverse sheaf in $D^b_{\str}(X, \bk)$ and $\vartheta_{\fff}$ an isomorphism 
    $$ \fll^T_{t}\boxtimes \fff \xrightarrow{\sim} a^\ast(\fff) $$
    satisfying 
    $$\label{cocycle iso1}(m\times \id_X)^\ast(\vartheta_{\fff}) =(\id_T\times a)^\ast(\vartheta_{\fff}) \circ (\fll^T_{t}\boxtimes \vartheta_{\fff}) \quad \mbox{and} \quad (\vartheta_{\fff})_{\mid\{1\}\times X}\cong \id_{\fff}.
$$
and whose morphisms $f : (\fff, \vartheta_{\fff}) \rightarrow (\fgg, \vartheta_{\fgg})$ are given by morphisms $f : \fff \rightarrow \fgg$ in $D^b_{\str}(X, \bk)$ such that
$$\xymatrix{
 \fll^T_{t}\boxtimes \fff \ar[rr]^{\fll^T_{t}\boxtimes f} \ar[d]^\wr_{\vartheta_{\fff}} && \fll^T_{t}\boxtimes \fgg  \ar[d]_\wr^{\vartheta_{\fgg}} \\
a^\ast(\fff) \ar[rr]_{a^\ast(f)} && a^\ast(\fgg)
}
$$ is a commutative diagram, 
    \item the full subcategory $\Piso(X, \bk)$ of perverse sheaves in $D^b_{\str}(X, \bk)$ such that there exists an isomorphism $a^\ast(\fff) \cong \fll^T_{t}\boxtimes \fff$.
    
\end{enumerate}
Our aim is to show that these four definitions agree, in other words, that the four above categories are equivalent. 
The proof of this fact will be given in a succession of lemmas; we will construct several natural functors between these categories. More precisely, we have the following natural functors:
\begin{enumerate}
        \item the functor $\For_{\small{\fll^T_t}} : \mathrm{P}^{\mathrm{LY}}_{T,\str}(X, \bk)_{\fll^T_{t}} \rightarrow \Perv(X, \bk),$
        \item the functor $\For_{\mathrm{co}}: \Pco (X, \bk)\rightarrow \Piso(X, \bk)$ that maps a pair $(\fff, \vartheta_{\fff})$ to $\fff$ and a morphism $f$ to itself,
        \item the functor $\For_{\mathrm{iso}} : \Perv(X, \bk)_{[\underline{t}]}\rightarrow \Perv(X, \bk)$ that maps an object to itself and a morphism to itself.
\end{enumerate}
Let us explicit the definition of $\For_{\fll^T_t}$: this is the composition of the functor $\For : D^{\mathrm{LY}}_{\str, T}(X,\bk)_{\fll^T_t} \rightarrow D^b_{\str}(X, \bk)$ with the shift $[r]$. According to definition \ref{definition pervers LY}, this functor indeed preserve perverse sheaves.
We will show that they induce equivalences between our different monodromic perverse categories. 

 Note that the local system $\fll^T_{t}$ has a stalk at $\{1\}$ equal to $\bk$. 
For $m$ divisible by the order of $t$, we have an isomorphism $\mathrm{can} : e_m^\ast(\fll^T_t) \cong \fbk_T$. In the next proof we will consider $T= \G_m$ and restriction of $\mathrm{can}$ to $e^{\frac{2i\pi}{m}}$. The reasoning of the proof of lemma \ref{monodromy des sytème locaux} allows one to see that we have $\mathrm{can}_{\mid e^{\frac{2i\pi}{m}}}= t^{-1}\cdot \id$.
More generally, for a cocharacter $\lambda \in \Cocar(T)$, one has $\mathrm{can}_{\mid \lambda_m} = \lambda(t)^{-1}\cdot \id$.

\begin{lem} \label{lemme monodromic = iso}
Consider $\fff \in \Perv(X, \bk)$. Then $\fff \in \Perv(X, \bk)_{[\underline{t}]}$ if and only if there exists an isomorphism 
$$a^\ast(\fff) \cong \fll^T_{t}\boxtimes \fff$$
whose restriction to $\{1\}\times X$ identifies with $\id_{\fff}$.
\end{lem}

\begin{proof}
To show the ``only if" part, we assume that we have an object $\fff$ in $\Perv(X, \bk)_{[\underline{t}]}$ and we will construct locally an isomorphism as in the statement of the lemma (i.e. on an open covering). Now $a^\ast(\fff)$ is a shifted perverse sheaf (since $a$ is smooth with connected fibers, cf. \cite[Proposition 4.2.5]{BBD}) and $\fll^T_t \boxtimes \fff$ as well (cf. \cite[Proposition 4.2.8]{BBD}), so according to \cite[Corollaire 2.1.23]{BBD}, we will obtain a global isomorphism.

Consider the case $T= \C^\ast$.
For $\fff\in \Perv(X, \bk)_{[\underline{t}]}$, there exist an integer $m\in \Z_{>0}$ such that we have an isomorphism $\iota : \pr_2^\ast(\fff)\xrightarrow{\sim} a(m)^\ast(\fff)$ satisfying $\iota_{\mid\{1\}\times X}= \id_{\fff}$.
We can assume that $m$ is divisible by the order of $t$, so that we have the isomorphism $\mathrm{can} : \fbk \xrightarrow{\sim}e_m^\ast(\fll^T_{t})$. Thus we obtain an isomorphism 
$$\Hom( e_m^\ast(\fll^T_{t})\boxtimes\fff,a(m)^\ast(\fff)) \xrightarrow{\sim} \Hom(\pr_2^\ast(\fff), a(m)^\ast(\fff)), \quad f \mapsto  f\circ (\mathrm{can}\boxtimes \id).$$
Now since $\fff$ is perverse, the isomorphism $\iota$ is unique. More precisely, for a given $n$, if there exists an isomorphism $\pr_2^\ast(\fff)\xrightarrow{\sim} a(n)^\ast(\fff) $ whose restriction to $\{1\}\times X$ is the identity, then this isomorphism is unique (this is essentially because the functor $\pr_2^\ast$ is fully faithful on perverse sheaves, see \cite[Proposition 4.2.5]{BBD}). This implies that there exists a unique isomorphism $I : = \iota \circ (\mathrm{can}\boxtimes \id) $ in  $\Hom(e_m^\ast(\fll^T_{t})\boxtimes\fff,a(m)^\ast(\fff) )$ whose restriction to $\{1\}\times X$ identifies with the identity of $\fff$ (note that by definition the restriction of $\mathrm{can}$ to $\{1\}\subseteq T$ is $\id$). Consider now the automorphism of $\C^\ast\times X $ given by 
$$\xi : (z,x) \longmapsto (ze^{\frac{2i\pi}{m}},x).$$
We clearly have $(e_m \times \id_X) \circ \xi = (e_m \times X)$ and similarly, $a(m) \circ \xi = a(m)$.
The isomorphism $\xi^\ast(I) : e_m^\ast(\fll^T_{t})\boxtimes\fff \xrightarrow{\sim} a(m)^\ast(\fff) $ satisfies 
\begin{align*}
    \xi^\ast(I)_{\mid \{1\}\times \id_X} & = (\iota\circ (\mathrm{can}\boxtimes \id))_{\mid \{e^{\frac{2i\pi}{m}}\}\times X} \\
    & =  (t\cdot\id_{\fff})\circ (\mathrm{can}_{e^{\frac{2i\pi}{m}}}\boxtimes \id)  \\
    & = (t\cdot \id_{\fff}) \circ (t^{-1}\cdot \id_{\fff})\\
    & = \id_{\fff}.
\end{align*}
(Note that we forgot the canonical isomorphism $e_m^\ast(\fll^T_{t})\boxtimes \fff\xrightarrow{\sim}\xi^\ast((e_m^\ast(\fll^T_{t})\boxtimes \fff)$ in the computation above, because its restriction to $\{1\}\times X$ is the identity; similarly for the identification $\xi^\ast a(m)^\ast \fff \cong a(m)^\ast \fff$.)
Thus we deduce that $\xi^\ast(I)$ coincides with $I=\iota\circ (\mathrm{can}\boxtimes \id)$. 

We will need a bit more of notation. Set 
$$O_1 = \{ z \in \C^\ast \mid z = r e^{i\theta}, \;\; r \in \R_{>0}, \; \theta \in (0,\frac{2\pi}{m})\}$$
and 
$$O_2 = \{ z \in \C^\ast \mid z = r e^{i\theta}, \;\; r \in \R_{>0}, \; \theta \in (-\frac{\pi}{m},\frac{\pi}{m})\}.$$
Then let $U_1 = \C^\ast \setminus \R_{\geq 0}$ and $U_2 = \C^\ast \setminus \R_{\leq 0}$. The map $e_m$ induces homeomorphisms 
$$O_1 \xrightarrow{\sim} U_1 \quad \mbox{and}\quad O_2 \xrightarrow{\sim} U_2.$$
We thus obtain isomorphisms 
{\small\begin{align*}
    \Hom(\pr_2^\ast(\fff)_{\mid O_i \times X}, a(m)^\ast (\fff)_{\mid O_i \times X}) & \cong  \Hom((e_m^\ast(\fll^T_{t})\boxtimes \fff)_{\mid O_i \times X}, a(m)^\ast (\fff)_{\mid O_i \times X}) \\
    &\cong \Hom( (\fll^T_{t}\boxtimes \fff)_{\mid U_i \times X}, a^\ast (\fff)_{\mid U_i \times X}).
    \end{align*}}
The first isomorphism is given by $(-)\circ (\mathrm{can}\boxtimes \id)_{\mid O_i \times X}$ and the second one by $\left(((e_m\times \id)_{\mid O_i\times X})^\ast\right)^{-1}$. 
    We denote by $I_i$ the image of $\iota_{\mid O_i\times X}$ obtained following these isomorphisms. By construction the restriction of the isomorphisms $I_1$ and $I_2$ to 
    $\{z\in \C^\ast \mid \Im(z) > 0\}$
    coincide. (Note that this is a connected component of $U_1 \cap U_2$.)
    For the other component 
    $$\{z\in \C^\ast \mid \Im(z) < 0\}$$
    we use the fact that $\xi^\ast(I)$ coincides with $I$ to see that there again the restriction of $I_1$ and $I_2$ coincide. We thus have ``constructed locally", i.e. on an open cover of $\C^\ast\times X$, an isomorphism between $a^\ast(\fff)$ and $\fll^T_{t}\boxtimes \fff$. Moreover these isomorphisms coincide on the intersection of the open subsets. We can then glue these ``local isomorphisms" to obtain a global isomorphism 
   \begin{equation} \label{isomorphismp tequiv}  \fll^T_{t}\boxtimes \fff \xrightarrow{\sim} a^\ast(\fff) .
   \end{equation}
    To determine the restriction of this isomorphism to $\{1\}\times X$, one can first restrict to $U_2\times X$ (since $1 \in U_2$). One then considers the pullback of \eqref{isomorphismp tequiv} along the homeomorphism $O_2\times X\cong U_2 \times X$. We obtain in this way the restriction of $\iota$ to $O_2\times X$. It is then easy to see that the restriction of \eqref{isomorphismp tequiv}
 to $\{1\}\times X$ gives the identity of $\fff$.
 
 To deal with the general case, one chooses a trivialization $T\cong (\C^\ast)^r$ and uses $\fbk_T \cong \fbk_{\C^\ast}\boxtimes \cdots \boxtimes \fbk_{\C^\ast}$ and $\fll^T_{t} \cong \fll^{\C^\ast}_{ t_1}\boxtimes \cdots \boxtimes \fll^{\C^\ast}_{ t_r}$ (for $t = (t_1, \ldots, t_r)$ under the induced trivialization $T^\vee_{\bk}\cong (\bk^\ast)^r$). We write $a_i : (\C^\ast)^i \times X\rightarrow(\C^\ast)^{i-1} \times X$ for $i= 1, \ldots, r$ with $a_i(t_1, \ldots, t_i, x)= (t_1, \ldots, t_{i-1}, a((1, \ldots, 1, \underset{i}{t_i}, 1, \ldots, 1), x))$, so that $a((t_1, \ldots, t_r), x) = a_1 \circ\cdots\circ a_r (t_1, \ldots, t_r, x)$. Applying the above reasoning successively to the $a_i$'s, we can conclude.
    
    Now we prove the converse. Assume that we have an isomorphism $a^\ast(\fff)\cong \fll^T_{t}\boxtimes \fff$ whose restriction to $\{1\}\times X$ is the identity of $\fff$ (under the identification $(\fll^T_{t})_1 \cong \bk$). Then, for $n$ the order of $t$, we obtain an isomorphism 
    $$\iota : e_n^\ast(\fll^T_{t})\boxtimes \fff \xrightarrow{\sim} a(n)^\ast(\fff) .$$
    We again have a canonical isomorphism 
    $  \mathrm{can}^{-1} : \fbk \xrightarrow{\sim}e_n^\ast(\fll^T_{t}).$ We can consider 
    $$\Tilde{\iota} : \fbk\boxtimes \fff \xrightarrow{\mathrm{can}^{-1 }\boxtimes \id_{\fff}} e_n^\ast(\fll^T_{t})\boxtimes \fff \xrightarrow{\iota} a(n)^\ast(\fff);$$
    the restriction of this isomorphism to $\{1\}\times X$ identifies with the identity of $\fff$. Thus the monodromy of $\fff$ is given by the restriction of $\Tilde{\iota}$ to various $\lambda_n = \lambda(e^{\frac{2i\pi}{n}})$ for $\lambda \in \Cocar(T)$. From the argument above, one sees that this restriction coincides with the restriction of $\mathrm{can}^{-1 }$ to $\{\lambda_n\}\times X$, which identifies with $\lambda(t)\cdot \id_{\fff}.$ This concludes the proof.
\end{proof}

In the following lemma we fix an isogeny $\widetilde{T}\xrightarrow{\nu} T$ in order to realize $D^{\mathrm{LY}}_{T, \str}(X, \bk)_{\small{\fll^T_t}}$ as a full subcategory of $D^b_{\widetilde{T}, \str}(X, \bk)$. Namely, we choose the map $e_n : T \rightarrow T, \; x \mapsto x^n$, for $n$ prime to $\ell$ and divisible by the order of $t$. Recall that we have a forgetful functor $\For: D^{\mathrm{LY}}_{T, \str}(X, \bk)_{\small{\fll^T_t}}\rightarrow D^b_{\str}(X, \bk)$, as introduced at the end of subsection \ref{section fully faithfulness pnu}. 

\begin{lem} \label{lemme for ff}
The restriction $\For_{\fll^T_{t}}$ of the functor $\For$ to $\mathrm{P}^{\mathrm{LY}}_{T,\str}(X, \bk)_{\fll^T_{t}}$ is fully faithful and has essential image contained in $\Perv(X, \bk)_{[\underline{t}]}$.
\end{lem}

\begin{proof}
The equivariant derived category is defined with respect to the action $t\cdot x = t^n x$ with $t \in T$ and $x\in X$. This action morphism $T\times X \rightarrow X$ is denoted $a_X(n) = a_X \circ (e_n \times \id_X)$. We will also consider the projection $\pr_X : T \times X \rightarrow X$.

It is well know (see e.g. \cite[Proposition A.2]{BaR}) that the forgetful functor $ \For : D^b_{\Tilde{T}, \str}(X, \bk) \rightarrow D^b_{\str}(X, \bk)$ induces an equivalence between the category of equivariant perverse sheaves and the full subcategory of constructible perverse sheaves whose objects are those $\fff$ such that there exists an isomorphism 
$$ \pr_X^\ast(\fff)\xrightarrow{\sim}
a_X(n)^\ast(\fff).$$
The functor $\For_{\fll^T_{t}}$ is the restriction to a direct summand subcategory of a fully faithful functor, we thus know that it is fully faithful. We need to show that its essential image is contained in the category $P_{\str}(X, \bk)_{[\underline{t}]}$. We will show that the restriction of $\For$ to $D^{\mathrm{LY}}_{T, \str}(X, \bk)_{\small{\fll^T_t}}$ has essential image in $D^b_{\str}(X, \bk)_{[\underline{t}]}$. As $\For$ maps perverse objects to perverse objects, this fact implies the lemma.

Consider $\fff\in D^{\mathrm{LY}}_{T,\str}(X, \bk)_{\fll^T_{t}}$. As this is a $\widetilde{T}$-equivariant sheaf, we have a canonical isomorphism (in the constructible category)
$$\vartheta_{\fff} : \pr_X^\ast(\For(\fff)) \xrightarrow{\sim} a_X(n)^\ast(\For(\fff)) $$
and this isomorphism satisfies the usual cocycle condition. We have two uses of this isomorphism: on the one hand, the condition $(\vartheta_{\fff}) _{\mid \{1\}\times X} = \id_{\fff}$ allows us to define the canonical monodromy morphism of $\fff$ from $\vartheta_{\fff}$: using the notations of subsection \ref{sectiopn def }, we have $\varphi_{\fff}^{\lambda} = \vartheta_{\fff, \lambda_n}$. 

On the other hand, the kernel $K = \ker(e_n)$ acts on $\fff$ via the automorphisms $\vartheta_{\fff, a}$ for $a\in K$ (see the proof of lemma \ref{lemme A agit sur Id}). 
One should remark that for any $\lambda \in \Cocar(T)$, we have $\lambda_n \in K$. Thus we have a commutative diagram 
$$\xymatrix{
\Cocar(T)\ar[rd]_{\varphi_{\fff}}  \ar[rr]^{ev(e^{\frac{2i\pi}{n}})} && K \ar[ld]^{\vartheta_{\fff, \cdot}}\\
& \Aut(\fff). &}
$$

By definition of the category $D^{\mathrm{LY}}_{T,\str}(X, \bk)_{\fll^T_{t}}$ the map $\vartheta_{\fff, (\cdot)}$ factors through the character $\chi_{\small{\fll^T_t}} : K \rightarrow \bk^\ast$, and it is clear that the horizontal map factors through the canonical projection $\Cocar(T) \twoheadrightarrow \Cocar(T)/e_n^\#(\Cocar(T))$. Using the proof of lemma \ref{isotypic component of the pushforward}, we get that the composition 
$$\Cocar(T) \twoheadrightarrow \Cocar(T)/e_n^\#(\Cocar(T))\rightarrow K \rightarrow \bk^\ast$$
coincides with the evaluation at $t\in T^\vee_{\bk}$. This tells us that the canonical morphism of monodromy of $\fff$ is given by
$$\varphi_{\fff}^\lambda = \lambda(t) \id_{\fff}$$
for any $\lambda \in \Cocar(T)$. In others words, we have $\For_{\fll^T_{t}}(\fff) \in D^b_{\str}(X, \bk)_{[\underline{t}]}$. 
\end{proof}

\begin{cor} \label{corollary noyau monodromy}
We keep the setting and notations of the proof of lemma \ref{lemme for ff}. For any $\fff\in D^b_{\Tilde{T}, \str}(X, \bk)$ and any $\lambda \in \Cocar(T)$, we have
$\vartheta_{\fff, \lambda_n} = \varphi_{\For(\fff)}^{\lambda}.$
\end{cor}

\begin{lem} \label{lemme LY = monodromic}
The functor 
$\For_{\fll^T_{t}} : \mathrm{P}^{\mathrm{LY}}_{T,\str}(X, \bk)_{\fll^T_{t}} \rightarrow \Perv(X, \bk)_{[\underline{t}]}$ is an equivalence of categories.
\end{lem}

\begin{proof}
We consider once again the isogeny $\widetilde{T}=T \xrightarrow{e_n}T$ and view $D^{\mathrm{LY}}_{T, \str}(X, \bk)_{\small{\fll^T_t}}$ as a full subcategory of $D^b_{\widetilde{T}, \str}(X, \bk)$.
Thanks to lemma \ref{lemme for ff}, we just have to show that the functor in the statement is essentially surjective. Consider an object $\fff \in \Perv(X, \bk)_{[\underline{t}]}$. By lemma \ref{lemme monodromic = iso}, we have an isomorphism $\fll^T_{t}\boxtimes \fff \xrightarrow{\sim}a^\ast(\fff)$ whose restriction to $\{1\}\times X$ is the identity of $\fff$. Pulling back along $(e_n\times \id_X)$ we obtain (as in the end of the proof of lemma \ref{lemme monodromic = iso}) an isomorphism 
$$\pr_2^\ast(\fff) \xrightarrow{\sim}a(n)^\ast \fff .$$
Thanks to \cite[Proposition A.2]{BaR}, we know that there exists a perverse object $\fgg$ in $D^b_{\Tilde{T}, \str}(X, \bk)$ such that $\For(\fgg) = \fff$. The main problem is to determine the action of $K = \ker(e_n)$ on $\fgg$. But thanks to corollary \ref{corollary noyau monodromy}, the action of $K$ on $\fgg$ gives (after applying $\For$) the monodromy of $\fff$. This tells us that $K$ acts via the character of evaluation at $t$ and concludes the proof.
\end{proof}

\begin{lem} \label{lemme iso= iso+1}
Consider $\fff\in \mathrm{P}_{\str}(X, \bk)$ such that there exists an isomorphism 
$$\fll^T_{t}\boxtimes \fff \xrightarrow{\sim}a^\ast(\fff) .$$
Then there exists a unique isomorphism $\vartheta :  \fll^T_{t}\boxtimes \fff\xrightarrow{\sim}a^\ast(\fff)$ such that $\vartheta_{\mid \{1\}\times X}$ identifies with the identity of $\fff$.
\end{lem}

\begin{proof}
We first show that $F = \fll^T_{t}\boxtimes (\cdot) : D^b_{\str}(X, \bk) \rightarrow D^b_{\Tilde{\str}}(T\times X, \bk)$ induces a fully faithful functor when restricted to perverse sheaves. (Here, $\Tilde{\str}$ is the induced stratification $\{T\times S \mid S \in \str\}$ on $T\times X$.)
Tensoring on the left with an invertible local system on $T\times X$ gives an auto-equivalence of category of $D^b_{\Tilde{\str}}(T\times X, \bk)$. Thus the functor $(\pr_1^\ast \fll^T_{t^{-1}}\otimes_{\bk} (-))\circ F$ is fully faithful if and only if $F$ is. This composition is clearly isomorphic to $\pr_2^\ast$, and the latter functor is fully faithful on perverse sheaves thanks to \cite[Proposition 4.2.5]{BBD}. Thus the restriction of $F$ to perverse sheaves is fully faithful. The inverse equivalence is induced by the functor $i^\ast$ where $i : X \rightarrow T \times X$ is the map that send $x$ to $(1, x)$. 

Now assume that $\beta : \fll^T_{t}\boxtimes \fff \xrightarrow{\sim}a^\ast(\fff)$ is any isomorphism, with $\fff$ perverse. Then, it is easy to see that $\vartheta := F(i^\ast(\beta)^{-1})\circ\beta$ is an isomorphism $\fll^T_{t}\boxtimes \fff \xrightarrow{\sim}a^\ast(\fff)$ such that $i^\ast(\vartheta) = \id_{\fff}$. 
If $\vartheta_1, \vartheta_2$ are two such isomorphisms, then $\vartheta_1\circ \vartheta_2^{-1}$ gives an automorphism of $\fll^T_{t}\boxtimes \fff$. Denote it by $F(g)$ (we can do that since $F$ is full). We have $i^\ast F(g) = g= \id_{\fff}$.
This readily implies that $\vartheta_1= \vartheta_2$ and concludes the proof of unicity, and hence the proof of the lemma.
\end{proof}

\begin{lem} \label{lemme iso = cocycle}
Consider $\fff\in \mathrm{P}_{\str}(X, \bk)$. Assume that there exists an isomorphism 
$$\vartheta_{\fff} : \fll^T_{t}\boxtimes \fff \xrightarrow{\sim} a^\ast(\fff) $$
such that $(\vartheta_{\fff})_{\mid \{1\}\times X}$ is $\id_{\fff}$.
Then we have 
\begin{equation}\label{cocycle iso}(m\times \id_X)^\ast(\vartheta_{\fff}) =  (\id_T\times a)^\ast(\vartheta_{\fff}) \circ (\fll^T_{t}\boxtimes \vartheta_{\fff}).
\end{equation}
Moreover for $(\fgg, \vartheta_{\fgg})$ another pair (perverse sheaf, isomorphism) and any map $f : \fff \rightarrow \fgg$, we have a commutative diagram 
$$\xymatrix{
 \fll^T_{t}\boxtimes \fff \ar[rr]^{\fll^T_{t}\boxtimes f} \ar[d]^\wr_{\vartheta_{\fff}} && \fll^T_{t}\boxtimes \fgg  \ar[d]_\wr^{\vartheta_{\fgg}} \\
a^\ast(\fff) \ar[rr]_{a^\ast(f)} && a^\ast(\fgg).
}
$$
\end{lem}

\begin{proof}
Rewrite the equality \eqref{cocycle iso} as $u = v$. We consider $u\circ v^{-1}$. This is an automorphism of $\fll^T_{t}\boxtimes \fll^T_{t}\boxtimes \fff$. Denote by $j$ the map $X \rightarrow T\times T \times X$ that sends $x$ to $(1,1,x)$. Using the same argument as in the proof of lemma \ref{lemme iso= iso+1}, we see that since $\fff$ is perverse, the maps $$\End_{D^b_{\str}(X, \bk)}(\fff, \fff) \leftrightarrow \End_{D^b_{\Tilde{\str}}(T\times T\times X, \bk)}((\fll^T_{t})^{\boxtimes 2}\boxtimes \fff, (\fll^T_{t})^{\boxtimes 2}\boxtimes \fff)$$
induced by the functors 
$$\fll^T_{t}\boxtimes\fll^T_{t}\boxtimes (\cdot) \quad \mbox{and} \quad j^\ast$$
are mutually inverse isomorphisms. This implies in particular that $u \circ v^{-1} = \id_{(\fll^T_{t})^{\boxtimes 2}\boxtimes \fff}$. Thus we have $u = v$.

The proof of the statement about morphisms is similar to the end of the proof of lemma \ref{lemme iso= iso+1}, and left to the reader.

\end{proof}

\subsection{The equivalences}

We can refine the definition of the functors $\For_{\small{\fll^T_t}}$ and $\For_{\mathrm{iso}}$. Lemma \ref{lemme monodromic = iso} tells us that $\For_{\mathrm{iso}}$ has essential image in $\Piso(X, \bk)$.
According to lemma \ref{lemme for ff}, $\For_{\small{\fll^T_t}}$ induces a functor $\mathrm{P}^{\mathrm{LY}}_{T,\str}(X, \bk)_{\fll^T_{t}}\rightarrow \Perv(X, \bk)_{[\underline{t}]}$. 
In fact, the results of section \ref{section Perverse monodromic sheaves}, give us a bit more: we have the following announced result.
\begin{prop} \label{proposition monodromic perverse sheaves coincide}
The four categories $\mathrm{P}^{\mathrm{LY}}_{T,\str}(X, \bk)_{\fll^T_{t}}$, $\Perv(X,\bk)_{[\underline{t}]}$, $\Pco(X, \bk)$ and $\Piso(X, \bk)$ are canonically equivalent: the functors $\For_{\small{\fll^T_t}}$, $\For_{\mathrm{co}}$ and $\For_{\mathrm{iso}}$ are equivalences of categories.
\end{prop}
\begin{proof}
All the work is already done: by lemma \ref{lemme monodromic = iso}, we have a canonical equivalence $\Perv(X, \bk)_{[\underline{t}]}\cong \Piso(X, \bk)$ (in fact these two categories are the same as full subcategories of $D^b_{\str}(X, \bk)$). From lemma \ref{lemme iso = cocycle} and lemma \ref{lemme iso= iso+1}, the categories $\Piso(X, \bk)$ and $\Pco(X, \bk)$ are equivalent. Finally, lemma \ref{lemme LY = monodromic} gives the equivalence $\mathrm{P}^{\mathrm{LY}}_{T,\str}(X, \bk)_{\fll^T_{t}}\cong \Perv(X, \bk)_{[\underline{t}]}$.
\end{proof}

\section{Highest weight structure of perverse monodromic sheaves} \label{section Highest weight structure of perverse monodromic sheaves}


\subsection{Standard, costandard and intersection cohomology complexes}


We consider the following setting: $\CX= \bigcup_{\alpha \in \Lambda}\CX_\alpha$ is an algebraically stratified $T$-variety with $T$ stable strata. In particular, $\Lambda$ is finite. We assume that there exists isomorphisms $\CX_{\alpha}\cong T \times \A^{n_\alpha}$ for any $\alpha$ (and we fix such isomorphisms once and for all) and that the action of $T$ on $ \CX_{\alpha}$ corresponds to the action of multiplication of $T$ on itself. We let $d_{\alpha} : = n_\alpha+ r$ be the dimension of $\CX_\alpha$. Let also $j_{\alpha} : \CX_{\alpha} \hookrightarrow \CX$ be the inclusion map.

The isomorphism classes of irreducible local systems on $\CX_\alpha$ are in bijection with elements of $T^\vee_{\bk}$: more precisely, the pullback along the map $\CX_{\alpha} \xrightarrow{\sim}T \times \A^{n_\alpha}\xrightarrow{\pr} T$ gives an equivalence $\loc(T, \bk) \cong \loc(\CX_{\alpha}, \bk)$. For any $t \in T^\vee_{\bk}$, we obtain an irreducible local system $\widetilde{\fll^{\alpha}_{t}}$. We then set $\fll^{\alpha}_{t} : = \widetilde{\fll^{\alpha}_{t}}[d_\alpha] $, so that $\fll^{\alpha}_{t}$ is perverse on $\CX_{\alpha}$.

We also denote 
$$\De^{\alpha}_{t} = (j_{\alpha})_! \fll^{\alpha}_t \quad \mbox {and} \quad \Na^{\alpha}_t := (j_{\alpha})_{\ast}\fll^{\alpha}_t.$$
Note that the $j_{\alpha}$'s are affine maps (indeed, the $\CX_{\alpha}$'s are affine and $\CX$ is a separated scheme), so that $\De^{\alpha}_{t}$ and $\Na^{\alpha}_t$ are perverse. Note also that thanks to lemma \ref{lemme pullback monodromy} and lemma \ref{monodromy des sytème locaux}, these objects are in $\mathrm{P}_{\Lambda}(\CX, \bk)_{[\underline{t}]}$.
Finally, let $\IC^{\alpha}_t : = (j_{\alpha})_{! \ast}\fll^{\alpha}_t$ be the intermediate extension of $\fll^{\alpha}_t$.

The $\De^{\alpha}_t$'s and $\Na^{\alpha}_t$'s for various $\alpha$ and $t$ are called standard and costandard objects respectively. (The justification for these names will be given in the next subsection.)

Note that the pullback under the isomorphism $\CX_{\alpha}\cong T \times \A^{n_\alpha}$ of the perverse sheaf $\fll^{\alpha}_t$ on $\CX_{\alpha}$ is the perverse sheaf $\fll^{T}_t[d_{\alpha}] \boxtimes \fbk_{\A^{n_\alpha}} $. Then the pullback of $\fll^{\alpha}_t$ under the action map $T \times \CX_{\alpha}\rightarrow \CX_{\alpha}$ is $\fll^{T}_t\boxtimes \fll^T_{t}[d_\alpha]\boxtimes \fbk_{\A^{n_\alpha}}$. Since for $n$ divisible by the order of $t$ we have $e_n^\ast \fll^T_{t} = \fbk_T$, it follows easily that $\fll^{\alpha}_{t}$ is $T$-equivariant for the action ``twisted" by $e_n$. It is then immediate that $\fll^{\alpha}_t[-r]$ is in $\mathrm{P}^{\mathrm{LY}}_{ T, \Lambda}(\CX_{\alpha}, \bk)_{\small{\fll^T_t}}$. It follows that $\De^{\alpha}_{t}[-r]$, $\Na^{\alpha}_t[-r]$ and $\IC^{\alpha}_t[-r]$ define objects in $\mathrm{P}^{\mathrm{LY}}_{ T, \Lambda}(\CX, \bk)_{\small{\fll^T_t}}$. In order to be consistent with \cite{LY}, we set 
$$\De^{\alpha}_t: = \De(\alpha)_{\small{\fll^T_t}}[-r], \quad \Na^{\alpha}_t: = \Na(\alpha)_{\small{\fll^T_t}}[-r], \quad \IC^{\alpha}_t: = \IC(\alpha)_{\small{\fll^T_t}}[-r]$$
when we consider these objects in $\mathrm{P}^{\mathrm{LY}}_{T,\Lambda}(X, \bk)_{\fll^T_{t}}$.
The forgetful functor $\mathrm{P}^{\mathrm{LY}}_{T,\Lambda}(\CX, \bk)_{\fll^T_{t}}\\ \rightarrow \mathrm{P}_{\Lambda}(\CX, \bk)_{[\underline{t}]}$ maps $\De(\alpha)_{\small{\fll^T_t}}$ to $\De^{\alpha}_t$ and similarly for the $\Na$'s and $\IC$'s.

\subsection{The case of a single stratum} \label{section single stratum}

 If we apply the equivalence of proposition \ref{proposition monodromic perverse sheaves coincide} to the case where $\CX = \CX_{\alpha}$ (i.e. a stratified space with only one stratum), we obtain an equivalence 
 $$\mathrm{P}^{\mathrm{LY}}_{T,\Lambda}(\CX_\alpha, \bk)_{\fll^T_{t}} \longleftrightarrow \mathrm{P}_{\Lambda}(\CX_{\alpha}, \bk)_{[\underline{t}]}.$$
Once again, the shifted pullback along the map $\CX_{\alpha}\xrightarrow{\sim}T \times \A^{n_\alpha}\rightarrow T$ gives an equivalence $\mathrm{P}_{\Lambda}(\CX_\alpha, \bk) \cong \loc(T, \bk)$ (since $\loc(\CX_\alpha, \bk) \cong \loc(T, \bk)$) and obviously 
 $$ \mathrm{P}_{\Lambda}(\CX_\alpha, \bk)_{[\underline{t}]} \cong \loc(T, \bk)_{[\underline{t}]}.$$
 We showed previously (see lemma \ref{monodromy des sytème locaux}) that for local systems the canonical morphism of monodromy (which defines an action of $\bk[\Cocar(T)]$) coincides with the action of $\bk[\Cocar(T)]$ on the associated $\bk[\Cocar(T)]$-module. In particular, we have an equivalence 
 $$ \loc(T, \bk)_{[\underline{t}]} \cong \bk[\Cocar(T)]/\langle e^\lambda - \lambda(t)\rangle \modu \cong \bk\modu.$$
 As $\bk$ is a field, this last category is semi-simple, and we deduce that $\mathrm{P}^{\mathrm{LY}}_{T,\Lambda}(\CX_\alpha, \bk)_{\fll^T_{t}}$ is semi-simple as well.

\subsection{Highest weight categories}
We recall here the definition of a highest weight category. Our main reference is \cite[\S 7]{R1}. Let $\mathcal{A}$ be a finite-length $\K$-linear abelian category with $\K$ a field such that for any objects $M,N \in \mathcal A$, the $\K$-vector-space $\Hom_{\mathcal{A}}(M,N)$ is finite-dimensional. Let $\str$ be the set of isomorphism classes of irreducible objects in $\mathcal{A}$ and assume that we have a partial order $\leq$ on $\str$. For any $s \in \str$, fix a representative $L_s$. Assume also that we have objects $\De_s, \Na_s$ together with morphisms 
$$\De_s \rightarrow L_s \quad \mbox{and} \quad L_s \rightarrow \Na_s.$$
For any subset $\str'\subseteq \str$, let $\mathcal{A}_{\str'}$ be the Serre subcategory generated by the $L_s$ with $s \in \str'$. We say that $\str'$ is an ideal for the order $\leq$ if ($s \in \str'$ and $t \leq s$) implies ($t \in \str'$).
\begin{defn} \label{definition hwc}
The data $(\mathcal{A}, (\str, \leq), L_s, \De_s\rightarrow L_s, L_s \rightarrow \Na_s)$ defines a highest weight category with weight poset $(\str, \leq)$ if the following conditions hold: 
\begin{enumerate}
    \item for any $s \in \str$, the subset $\{t\leq s\mid t \in \str\}$ is finite, 
    \item for each $s\in \str$, we have $\Hom_{\mathcal{A}}(L_s, L_s) = \K$, 
    \item for any ideal $\str' \subseteq \str$ such that $s$ is maximal in $\str'$, $\De_s \rightarrow L_s$ is a projective cover in $\mathcal{A}_{\str'}$ and $L_s \rightarrow \Na_s$ is an injective hull in $\mathcal{A}_{\str'}$, 
    \item the kernel of $\De_s \rightarrow L_s$ and the cokernel of $L_s \rightarrow \Na_s$ are in $\mathcal{A}_{<s}$, 
    \item for any $s,t\in \str$, we have $\Ext^2_{\mathcal{A}}(\De_s, \Na_t)=0$.
\end{enumerate}
The objects $\De_s$ for $s \in \str$ are called standard objects; the $\Na_s$'s are called costandard.
\end{defn}
We almost always forget part of the data above and just refer to $\mathcal{A}$ as a highest weight category.

\subsection{Highest weight structure}
This section and the arguments that follow are faithfully inspired from \cite[\S 3.2 and 3.3]{BGS}. 
The set $\Lambda$ is a poset for the order $\leq$ defined by
$$(\alpha \leq \alpha')\; \Leftrightarrow \;(\overline{\CX_{\alpha}} \subseteq \overline{\CX_{\alpha'}}).$$

 \begin{thm}
 The category $\mathrm{P}_{\Lambda}(\CX, \bk)_{[\underline{t}]}$ together with the natural morphisms $\De^{\alpha}_t \twoheadrightarrow \IC^{\alpha}_t \hookrightarrow \Na^{\alpha}_t$ and with weight poset the pair $(\Lambda, \leq)$ is a highest-weight category.
 \end{thm}

 \begin{proof} For the sake of simplicity (of notations) we will set, for the duration of this proof, $\mathrm{P}_{\Lambda}(\CX,\bk)_{[\underline{t}]} = \mathcal{C}$. 
 According to definition \ref{definition hwc}, we have 5 points to check. The first one is clear (since $\Lambda$ itself is finite). Using the fact that $(j_\alpha)_{!\ast} : \loc(\CX_\alpha, \bk) \rightarrow \mathrm{P}_{\Lambda}(\CX, \bk)$ is fully faithful and the equivalence $\loc(\CX_\alpha, \bk) \cong \loc(T, \bk)$, one gets
 $$\Hom_{\mathcal{C}}(\IC^{\alpha}_{t}, \IC^{\alpha}_{t})=\Hom_{\mathrm{P}_{\Lambda}(\CX, \bk)}(\IC^{\alpha}_{t}, \IC^{\alpha}_{t}) = \Hom_{\loc(T, \bk)}(\fll^T_{t}, \fll^T_{t}) = \bk.$$

 Now consider an ideal $I \subset \Lambda$. Assume that $\alpha$ is a maximal element in $I$. We have to show that the canonical morphism $\De^{\alpha}_{t}\rightarrow \IC^{\alpha}_{t}$ is the projective cover of $\IC^{\alpha}_{t}$ in the category $\mathcal{C}_{I}$. We must show that 
\begin{equation}\label{hwt 1}\Hom_{\mathcal{C}_I}(\De^{\alpha}_{t}, \IC^{\beta}_{t}) = \left\{
\begin{array}{rl}
  \bk &\mbox{ if } \alpha = \beta\\
0 &\mbox{ otherwise}
\end{array}
\right. 
\end{equation}
and that 
\begin{equation} \label{hwt 2} \Ext^1_{\mathcal{C}_I}(\De^{\alpha}_{t}, \IC^{\beta}_{t}) =0
\end{equation}
for any $\beta \in I$.
We use the equivalence 
$$\mathrm{P}_{\Lambda}(\CX, \bk)_{[ \underline{t}]} \cong \mathrm{P}^{\mathrm{LY}}_{T,\Lambda}(\CX_\alpha, \bk)_{\fll^T_{t}}$$
of lemma \ref{lemme LY = monodromic}. In the latter category, we have $\Ext^1_{\mathrm{P}^{\mathrm{LY}}_{T, \Lambda}(X, \bk)_{\small{\fll^T_t}}}(\De(\alpha)_{\small{\fll^T_t}}, \IC(\beta)_{\small{\fll^T_t}}) = \Hom_{D^{\mathrm{LY}}_{T, \Lambda}(X, \bk)_{\small{\fll^T_t}}}\De(\alpha)_{\small{\fll^T_t}}, \IC(\beta)_{\small{\fll^T_t}}[1])$ (see \cite[Remarque 3.1.17 (ii)]{BBD}, this is a standard feature of $t$-structure). We can then use adjunction and we deduce \eqref{hwt 1} and \eqref{hwt 2} in the case $\alpha \neq \beta$. We treat the case $\alpha = \beta$. Using the equivalence above two times and adjunction, we obtain an isomorphism:
$$\Ext^1_{\mathrm{P}_{\Lambda}(\CX_\alpha, \bk)_{[ \underline{t}]}}(\fll^{\alpha}_{t} , \fll^{\alpha}_{t}) \cong \Ext^1_{\loc(T, \bk)_{[\underline{t}]}}(\fll^T_{t}, \fll^T_{t}) .$$ Now the category $\loc(T, \bk)_{[\underline{t}]}$ is semisimple as we saw in \S \ref{section single stratum}.
so this $\Ext$ space vanishes. This concludes the verification of \eqref{hwt 2}.
We check similarly that $\IC^{\alpha}_{t} \rightarrow \Na^{\alpha}_{t}$ is an injective hull in $\mathcal{C}_I$. 

The fact that the kernel (resp. the cokernel) of $\De^{\alpha}_{t}\rightarrow \IC^{\alpha}_{t}$ (resp. of $\IC^{\alpha}_{t} \rightarrow \Na^{\alpha}_{t}$) lies in $\mathcal{C}_{< \alpha}$ comes from the facts that these objects are supported on $\overline{\CX_\alpha} = \sqcup_{\beta \leq \alpha}\CX_\beta$ and that this morphism is an isomorphism once restricted to $\CX_\alpha$.  

We check that $\Ext^2_{\mathcal{C}}(\De^{\alpha}_{t}, \Na^{\beta}_{t})=0$ for any $\alpha, \beta \in \Lambda$. 
We have 
$$\Ext^2_{\mathrm{P}_{\Lambda}(\CX, \bk)_{[ \underline{t}]}}(\De^{\alpha}_{t}, \Na^{\beta}_{t}) \cong \Ext^2_{\mathrm{P}^{\mathrm{LY}}_{T,\Lambda}(\CX_\alpha, \bk)_{\fll^T_{t}}}(\De(\alpha)_{\fll^T_{t}}, \Na(\beta)_{\fll^T_{t}}).$$
Now we know $\mathrm{P}^{\mathrm{LY}}_{T,\Lambda}(\CX_\alpha, \bk)_{\fll^T_{t}}$ is the heart of the perverse $t$-structure on the category\\ $D^{\mathrm{LY}}_{ T, \Lambda}(\CX, \bk)_{\fll^T_{t}}$. Thanks to \cite[Remarque 3.1.17 (ii)]{BBD}, we have an injection 
$$\Ext^2_{\mathrm{P}^{\mathrm{LY}}_{T,\Lambda}(\CX_\alpha, \bk)_{\fll^T_{t}}}(\De(\alpha)_{\fll^T_{t}}, \Na(\beta)_{\fll^T_{t}}) \hookrightarrow \Hom_{D^{\mathrm{LY}}_{ T, \Lambda}(\CX, \bk)_{\fll^T_{t}}}(\De(\alpha)_{\fll^T_{t}}, \Na(\beta)_{\fll^T_{t}}[2]).$$ We use adjunction; the only case that still needs work is the case $\alpha = \beta$. What we need to consider is now 
$$ \Hom_{D^{\mathrm{LY}}_{ T, \Lambda}(\CX_{\alpha}, \bk)_{\fll^T_{t}}}(\fll^{\alpha}_{t}, \fll^{\alpha}_{t}[2]).$$
Recall that we denoted $K_n$ the kernel of the n-th power map $T \rightarrow T$. We have an isomorphism $T \cong \widetilde{T}/K \cong \widetilde{T}\times^{K_n} \{\pt\}$ where the last space is the quotient under the action of $K_n$ by multiplication on $\widetilde{T}$. We also denote by $\Rep_{\bk}(K_n)$ the category of finite dimensional $\bk$-representation of $K_n$.
One has the following sequence of equivalences of categories: 
$$ D^b_{\Tilde{T}}(T, \bk) \cong D^b_{\Tilde{T}}(\Tilde{T}\times^{K_n}\{\pt\}, \bk)\cong D^b_{K_n}(\{\pt\}, \bk)\cong D^b(\Rep_{\bk}(K_n)).$$
and this last category is the derived category of a semisimple abelian category (since the cardinality of $K_n$ is prime to $\ell$). Now $D^{\mathrm{LY}}_{ T, \Lambda}(\CX_{\alpha}, \bk)_{\fll^T_{t}}$ is a full subcategory of $D^b_{\Tilde{T}, \Lambda}(\CX_{\alpha}, \bk)$ and the pullback functor $D^b_{\Tilde{T}, \Lambda*ù}(T, \bk)\rightarrow D^b_{\Tilde{T}, \Lambda}(\CX_\alpha, \bk)$ is fully faithful, so
we see that the above $\Hom$-space vanishes and we get the result.
 \end{proof}

\begin{cor}
The realization functor $D^b\mathrm{P}^{\mathrm{LY}}_{T, \Lambda}(\CX, \bk)_{\small{\fll^T_t}}\rightarrow D^{\mathrm{LY}}_{T, \Lambda}(\CX, \bk)_{\small{\fll^T_t}}$ is an equivalence of categories.
\end{cor}
\begin{proof}
This can be proved using the arguments of \cite[Corollaries 3.2.2 and 3.3.2]{BGS}.
\end{proof}

\end{document}